\documentclass[11pt]{article}
\usepackage{cite}
\usepackage{mathrsfs} 
\usepackage{ifthen}
\usepackage{latexsym}
\usepackage{amsmath,amssymb,amstext,amsthm} 
\usepackage{graphicx,color,epsfig}
\usepackage{graphicx}
\usepackage{fullpage}        
\usepackage{float}
\usepackage{verbatim}
\usepackage{enumerate}
\usepackage[pdftex,letterpaper=true,pagebackref=true]{hyperref}

\hypersetup{
    plainpages=false,       
    pdfpagelabels=true,     
    bookmarks=true,         
    unicode=false,          
    pdftoolbar=true,        
    pdfmenubar=true,        
    pdffitwindow=false,     
    pdfstartview={FitH},    
    pdftitle={
FR for SR SDP and DNN
},    
    pdfnewwindow=true,      
    colorlinks=true,        
    linkcolor=blue,         
    citecolor=green,        
    filecolor=magenta,      
    urlcolor=cyan           
}

\numberwithin{equation}{section}
\usepackage{makeidx}
\usepackage{lineno} 
\usepackage{algorithm}
\usepackage[noend]{algpseudocode}
\usepackage[noabbrev]{cleveref}
\usepackage{multirow}
\usepackage{rotating} 
\usepackage{enumitem}
\makeindex

\def\R{\mathbb{R}}
\def\Sc{\mathbb{S}}
\def\Ss{\mathbb{S}}
\def\Sn{\Sc^n}

\def\Snp{\Sc_+^n}

\def\Rd{\mathbb{R}^d}
\def\Rn{\mathbb{R}^n}

\def\Rq{\mathbb{R}^q}

\def\SSx{\mathcal{S}_x}

\def\eqref#1{{\normalfont(\ref{#1})}}

\def\PSD{\mbox{\boldmath$PSD$}\,}

\def\eqref#1{{\normalfont(\ref{#1})}}

\newtheorem{theorem}{Theorem}[section]

\newtheorem{definition}[theorem]{Definition}
\newtheorem{example}[theorem]{Example}
\newtheorem{assump}[theorem]{Assumption}
\newtheorem{prop}[theorem]{Proposition}

\newtheorem{lem}[theorem]{Lemma}

\newtheorem{corollary}[theorem]{Corollary}

\newtheorem{remark}[theorem]{Remark}

\newtheorem{alg}[theorem]{Algorithm}

\newtheorem{lemma}[theorem]{Lemma}

\crefname{thm}{Theorem}{Theorems}
\Crefname{thm}{Theorem}{Theorems}
\crefname{assump}{Assumption}{Theorems}
\Crefname{assump}{Assumption}{Theorems}
\crefname{question}{Question}{Theorems}
\Crefname{question}{Question}{Theorems}
\crefname{problem}{Problem}{Theorems}
\Crefname{problem}{Problem}{Theorems}
\crefname{conjecture}{Conjecture}{Theorems}
\Crefname{conjecture}{Conjecture}{Theorems}
\crefname{prop}{Proposition}{Propositions}
\Crefname{prop}{Proposition}{Propositions}
\crefname{cor}{Corollary}{Corollaries}
\Crefname{cor}{Corollary}{Corollaries}
\crefname{alg}{Algorithm}{Algorithms}
\Crefname{alg}{Algorithm}{Algorithms}
\crefname{lem}{Lemma}{Lemmas}
\Crefname{lem}{Lemma}{Lemmas}
\theoremstyle{definition}
\crefname{defn}{definition}{definitions}
\Crefname{defn}{Definition}{Definitions}
\crefname{remark}{Remark}{Remarks}
\Crefname{remark}{Remark}{Remarks}
\crefname{rmk}{Remark}{Remarks}
\Crefname{rmk}{Remark}{Remarks}
\crefname{example}{Example}{Examples}
\Crefname{example}{Example}{Examples}
\crefname{align}{}{}
\Crefname{align}{}{}
\crefname{equation}{}{}
\Crefname{equation}{}{}

	\newcommand{\textdef}[1]{\textit{#1}\index{#1}}

	\newcommand{\LL}{{\mathcal L} }
        \newcommand{\RR}{{\mathcal R} }
	
	\newcommand{\FF}{{\mathcal F} }

	\newcommand{\BB}{{\bf \mathcal B} }
	\newcommand{\BBs}{{\BB^*}}
	\newcommand{\BBt}{\tilde{\BB}}
	\newcommand{\BBts}{\BBt^*}
	
	\newcommand{\GG}{{\mathcal G} }
	\newcommand{\RRG}{{\RR_{\GG}}}
	
	\newcommand{\PP}{{\mathcal P} }
	\newcommand{\Nn}{{\NN}^n}
	\newcommand{\TT}{{\mathcal T} }
	
	\newcommand{\VV}{{\mathcal V} }

	\newcommand{\LB}{\textbf{LB}\,}
	\newcommand{\LBp}{\textbf{LB}}
	\newcommand{\OBJ}{\textbf{OBJ}\,}
	\newcommand{\OBJp}{\textbf{OBJ}}
	\newcommand{\ADMM}{\textbf{ADMM}\,}
	\newcommand{\ADMMp}{\textbf{ADMM}}
	
	\newcommand{\DNN}{\textbf{DNN}\,}
	\newcommand{\DNNp}{\textbf{DNN}}
	\newcommand{\GP}{\textbf{GP}\,}
	\newcommand{\GPp}{\textbf{GP}}
	\newcommand{\QAP}{\textbf{QAP}\,}
	\newcommand{\QAPp}{\textbf{QAP}}
	
	\newcommand{\IPMp}{\textbf{IPM}}
	\newcommand{\SDP}{\textbf{SDP}\,}
	\newcommand{\SDPp}{\textbf{SDP}}
	\newcommand{\SDPs}{\textbf{SDPs}\,}

	\newcommand{\SR}{\textbf{SR}\,}
	\newcommand{\SRp}{\textbf{SR}}
	\newcommand{\FR}{\textbf{FR}\,}
	\newcommand{\FRp}{\textbf{FR}}

	\newcommand{\Rm}{{\R^m\,}}

	\newcommand{\NN}{{\mathbb N}}

	\newcommand{\A}{{\mathcal A}}
	\newcommand{\At}{{\tilde \A}}
	\newcommand{\II}{{\mathcal I}}

	\newcommand{\F}{{\mathcal F\,}}

	\newcommand{\QQ}{{\mathcal Q\,}}

	\newcommand{\bbm}{\begin{bmatrix}}
	\newcommand{\ebm}{\end{bmatrix}}
	\newcommand{\bem}{\begin{pmatrix}}
	\newcommand{\eem}{\end{pmatrix}}
	\newcommand{\beq}{\begin{equation}}
	\newcommand{\beqs}{\begin{equation*}}
	\newcommand{\bet}{\begin{table}}
	\newcommand{\eeq}{\end{equation}}
	\newcommand{\eeqs}{\end{equation*}}
	\newcommand{\beqr}{\begin{eqnarray}}

	\DeclareMathOperator{\face}{face}
	\DeclareMathOperator{\sd}{sd}

	\DeclareMathOperator{\Null}{null}
	\DeclareMathOperator{\nul}{null}
	\DeclareMathOperator{\Range}{range}
	\DeclareMathOperator{\range}{range}
	
	\DeclareMathOperator{\dist}{dist}

	\DeclareMathOperator{\Tr}{{\mathrm{T}}}
	\DeclareMathOperator{\trace}{{trace}}
	\DeclareMathOperator{\aff}{{aff}}

	\DeclareMathOperator{\blkdiag}{{blkdiag}}
	\DeclareMathOperator{\Blkdiag}{{Blkdiag}}
	\DeclareMathOperator{\tr}{{trace}}
	\DeclareMathOperator{\diag}{{diag}}
	\DeclareMathOperator{\Diag}{{Diag}}

	\DeclareMathOperator{\offDiag}{{offDiag}}

	\DeclareMathOperator{\gsvec}{{gsvec}}
	\DeclareMathOperator{\svec}{{svec}}

	\DeclareMathOperator{\relint}{{relint}}

	\DeclareMathOperator{\rank}{{rank}}
	\DeclareMathOperator{\spanl}{{span}}
	\DeclareMathOperator{\aut}{{aut}}
	
	\DeclareMathOperator{\conv}{{conv}}

	\newcommand{\nc}{\newcommand}
	\nc{\arrow}{{\rm arrow\,}}
	\nc{\Arrow}{{\rm Arrow\,}}
	\nc{\BoDiag}{{\rm B^0Diag\,}}
	\nc{\bodiag}{{\rm b^0diag\,}}

	\nc{\Mm}{{\mathcal M}^{m} }
	\nc{\Mmn}{{\mathcal M}^{mn} }
	\nc{\Mnr}{{\mathcal M}_{nr} }
	\nc{\Mnmr}{{\mathcal M}_{(n-1)r} }
	\nc{\kwqqp}{Q{$^2$}P\,}
	\nc{\kwqqps}{Q{$^2$}Ps}

	\nc{\notinaho}{(X,S)\in \overline{AHO}(\A)}
	\nc{\inaho}{(X,S)\in AHO(\A)}
	
	\newcommand{\bea}{\begin{eqnarray}}%
	\newcommand{\eea}{\end{eqnarray}}%
	\newcommand{\beas}{\begin{eqnarray*}}%
	\newcommand{\eeas}{\end{eqnarray*}}%
	\renewcommand{\F}{\mathcal{F}}%
	{}
	
	\newcommand{\Hnp}[1][]{\,\mathbb{H}_+^{\ifthenelse{\equal{#1}{}}{n}{#1}}}
	\newcommand{\Hn}[1][]{\,\mathbb{H}^{\ifthenelse{\equal{#1}{}}{n}{#1}}}
	\newcommand{\Dn}[1][]{\,\mathbb{D}^{\ifthenelse{\equal{#1}{}}{n}{#1}}}
	
\begin{document}

	\bibliographystyle{plain}
	\title{
	Facial Reduction for Symmetry Reduced Semidefinite
\\and Doubly Nonnegative Programs
	}
	             \author{
	\href{https://huhao.org/}{Hao Hu}\thanks{
Department of Combinatorics and Optimization
	Faculty of Mathematics, University of Waterloo, Waterloo,
	Ontario, Canada N2L 3G1;
	\url{h92hu@uwaterloo.ca}.}
	\and
	\href{https://www.tilburguniversity.edu/staff/r-sotirov}{Renata Sotirov}\thanks{
\href{https://www.tilburguniversity.edu/about/schools/economics-and-management/organization/departments/eor}{Department
of Econometrics and Operations Research},
Tilburg University, The Netherlands; \url{r.sotirov@uvt.nl}}
	\and
	\href{http://www.math.uwaterloo.ca/~hwolkowi/}
	{Henry Wolkowicz}%
	\thanks{Department of Combinatorics and Optimization
	Faculty of Mathematics, University of Waterloo, Waterloo,
	Ontario, Canada N2L 3G1; Research supported by The Natural
	Sciences and Engineering Research Council of Canada;
	\url{www.math.uwaterloo.ca/\~hwolkowi}.
	}
	}
	
	          \maketitle
	
	
	
	\vspace{-.4in}
	
\begin{abstract}
We consider both facial reduction, \FRp, and symmetry reduction, \SRp, techniques for semidefinite programming, \SDPp.
We show that the two together fit surprisingly well in an alternating direction method of multipliers, \ADMMp, approach.
In fact, this approach allows for simply adding on nonnegativity constraints, and solving the doubly nonnegative, \DNN, relaxation of many classes of
hard combinatorial problems. We also show that the singularity degree remains the same
after \SRp, and that the \DNN relaxations considered here have singularity degree one, that is reduced to zero after \FRp.
The combination of \FR and \SR leads to a significant improvement in both numerical stability and running time for both the \ADMM and interior point approaches.

We test our method on various  \DNN relaxations of hard combinatorial problems including quadratic assignment problems with
sizes of more than $n=500$. This translates to a semidefinite constraint of order $250,000$ and $625\times 10^8$ nonnegative constrained variables,   before applying the reduction techniques.
\end{abstract}

	

	\tableofcontents
	\listoftables

	\vspace{-.2in}

	\section{Introduction}
	\label{sec:intro}
We consider two reduction techniques, facial and symmetry reduction, for
semidefinite programming, \SDPp. We see that the exposing vector approach
for \textdef{facial reduction, \FRp}, moves naturally onto the
\textdef{symmetry reduction, \SRp}.
We show that the combination of the two reductions fits surprisingly well
in an alternating direction method of multipliers, \ADMMp, approach.
In fact,  the combination of \FR and \SR makes possible solving the doubly nonnegative, \DNN, relaxations of many classes of hard combinatorial problems by using \ADMMp.
The combination of facial and symmetry reduction also leads
to a significant improvement in both numerical stability and
running time for both the \ADMM and interior point approaches.
We test our method on various \DNN relaxations of
hard combinatorial problems including quadratic assignment problems (\QAPp) with
sizes of more than $n=500$, see Table \ref{table:mittlem}.
Note that the order of the symmetric matrix variable in the
\SDP relaxation of the \QAP with  $n=500$,  before applying the reduction techniques, is $250,000$. This yields
approximately $625\times 10^8$ nonnegatively constrained variables in the semidefinite constrained matrix of the original, not reduced, problem formulation.
\index{\FRp, facial reduction}
\index{\SRp, symmetry reduction}
	
	Semidefinite programming can be viewed as an extension of linear
programming where the nonnegative orthant is replaced by the cone of
positive semidefinite matrices.
Although there are many algorithms for solving semidefinite programs,
they currently do not scale well and often do not provide high accuracy
solutions.
An early method for exploiting sparsity and reducing problem
size was based on recognizing a
chordal pattern in the matrices forming the \SDPp, see
e.g.,~\cite{FuFuKoNa:97a,kungurtsev2018two},
and the survey~\cite{VandenbergeAndersen:15}.  Another technique is that of
symmetry reduction, a methodology, pioneered by
Schrijver~\cite{schrijver1979comparison}, that exploits symmetries in
the data matrices that allows for the problem size to be reduced, often
significantly. More details and surveys for \SR are available in
\cite{MR2894697,MR2894695}
	
Without loss of generality, we consider the case where the primal
problem has a finite optimal value. Then for linear programming,
strong duality holds for both the primal and the dual problems.
But, this is \emph{not} the case
for \SDPp, where the primal and/or the dual can be unattained, and
one can even have a positive \emph{duality gap} between the
primal and dual optimal values. The usual constraint qualification to
guarantee strong duality is the Slater condition, strict feasibility.
Failure of the Slater condition may lead to theoretical and numerical
problems when solving the \SDPp. Facial reduction, \FRp, introduced by
Borwein and Wolkowicz~\cite{bw1,bw2,bw3}, addresses this issue by
projecting the \textdef{minimal face} of the \SDP into a lower dimensional
space. The literature for the theory and applications for \FR is large.
For a recent survey and theses
see~\cite{DrusWolk:16,permenter2017reduction,Sremac:2019}.

An earlier work~\cite{Lofberg} combines  partial  \FR and  \SR  for solving sum of square ({\bf SOS}) programs.
In particular, L\"{o}fberg~\cite{Lofberg}  applies a partial \FR via monomial basis selection
and shows  how to perform a partial \SR via identification of sign-symmetries to obtain  block-diagonal {\bf SOS} programs.
Examples in~\cite{Lofberg} verify the efficiency of the combined approach for {\bf SOS} programs.
For the connection  between \FR and monomial basis selection see~\cite{DrusWolk:16,Waki_mur_sparse}.

\medskip

In our paper, we assume that we know how to do \FR and \SR separately for the input \SDP instance. Under this
assumption, we show that it is possible to implement \FR to the symmetry
reduced \SDPp. The obtained reduced \SDP is both facially reduced and
symmetry reduced. And, it can be solved in a numerically stable manner by interior point methods.
Moreover, the nonnegativity constraints can be added to the
original \SDPp, and the resulting \DNN relaxation can be solved efficiently, as the
nonnegativities follow through, and are in fact simplified,
to the reduced \SDP program.
Thus, in fact we solve the facially and  symmetry reduced \DNN relaxation using an alternating direction method of
multipliers approach, \ADMMp. As a consequence, we are able to solve some  huge \DNN relaxations
for highly symmetric instances of certain hard combinatorial problems,
and we do so in a reasonable amount of time.

We include theoretical results on facial reduction,
as well as on the singularity degree of both \SDP and \DNN relaxations.
We present a view of \FR for \DNN from the ground set of the original
hard combinatorial problem. The singularity degree indicates the importance of \FR for splitting
type methods. In particular we show that the singularity
degree remains the same after  \SRp, and that our applications all have singularity degree one, that get reduced to zero after \FRp.
\index{doubly nonnegative, \DNNp}
\index{\DNNp, doubly nonnegative}

\subsection{Outline}
In~\Cref{sect:backgr} we provide
the background on using substitutions to first obtain
\FR and then symmetry and block diagonal \SRp.
In~\Cref{sect:FRforsymm} we
show how to apply \FR to the symmetry reduced \SDPp, and we
also provide conditions such that the obtained \SDP is strictly feasible.
In fact, we show that the nonnegativity constraints are essentially
unchanged and that we have strict feasibility
for the reduced \DNN relaxation. The results that singularity degree
does not change after \SR are included as well, see~\Cref{sect:newDNN}. In~\Cref{sec_admm} we show that the
reduced  \DNN relaxation can be solved efficiently using an \ADMM approach.
In~\Cref{sec_numers} we apply our result to two classes of problems:
the quadratic assignment and graph partitioning problems.
Concluding comments are in~\Cref{sect:concl}.

\section{Background}
\label{sect:backgr}
	\subsection{Semidefinite programming}
	The \textdef{semidefinite program, \SDPp}, in standard form is
	\index{\SDPp, semidefinite program}
	\index{primal optimal value, $p^*_{\SDP}$}
	\index{$p^*_{\SDP}$, primal optimal value}
\begin{equation}
	\label{sdp_standard}
p^*_{\SDP} :=	\min \{ \langle C,X \rangle \;|\; \textdef{$\A(X)$}=b, \;\;  X \succeq 0\},
	\end{equation}
	where the linear transformation $\A:\Sn \to \Rm$ maps real $n\times n$
	symmetric matrices to $\Rm$, and
 \textdef{$X\succeq 0$} denotes positive semidefiniteness, i.e., $X\in \Snp$.
 The set $\Snp$ is  the \emph{positive semidefinite, $\PSD$, cone}.
In the case of a  \textdef{doubly nonnegative, \DNNp}, relaxation,
nonnegativity constraints, $X\geq 0$, are added to~\cref{sdp_standard},
i.e.,~we use the \emph{\DNN cone} denoted $\DNNp\cong \DNN^n = \Snp\cap \R^{n\times n}_+$.
Without loss of generality, we assume that \underline{$\A$ is onto}.
	We let
\index{\DNN cone, $\DNN^n\cong \DNN$}
\index{$\DNN^n\cong \DNN$, \DNN cone}
 \index{\DNNp, doubly nonnegative}
	\index{$\Sn$, symmetric matrices}
	\index{symmetric matrices, $\Sn$}
	\index{$\Snp$, positive semidefinite cone}
	\index{positive semidefinite cone, $\Snp$}
	\index{$\succeq 0$, positive semidefinite}
	\index{positive semidefinite, $\succeq 0$}
	\index{$\PP_F$, feasible problem}
	\index{feasible problem$, \PP_F$}
	\index{$\FF_X$, feasible set}
	\index{feasible set$, \FF_X$}
	\begin{equation}
\label{eq:feasprob}
(\PP_F) \qquad	\FF_X:=\{X \succeq 0 \,|\,  \textdef{$\A(X)$}=b \}
	\end{equation}
	denote the \emph{feasibility problem}
     for this formulation with data $\A,b, \Snp$ from~\cref{sdp_standard}.
	Note that the linear equality constraint is equivalent to
	\[
	\A(X)= \left(\langle A_i,X \rangle\right)= (b_i)\in \Rm,
	\]
	for some $A_i\in \Sn, i=1,\ldots,m$.
The \textdef{adjoint} transformation ${\A^*:\Rm\to\Sn}$ is
$\textdef{$\A^*(y)$}=\sum_{i=1}^m y_iA_i$.
	
	\subsubsection{Strict feasibility and facial reduction}
The standard constraint qualification to guarantee strong
duality\footnote{Strong duality for the primal means a zero duality gap,
$p^*_{\SDP}=d^*_{\SDP}$, and dual attainment.}  for
the primal \SDP is the
	\textdef{Slater constraint qualification} (strict feasibility)
	\[
	\exists \hat X :
	         \, \,  \, \A(\hat X)=b, \, \, \hat X \succ 0,
	\]
where \textdef{$\hat X\succ 0$} denotes positive definiteness, i.e., $\hat X \in \Ss^n_{++}$.
For many problems where strict feasibility fails,
 one can exploit structure and facially reduce the
	problem to obtain strict feasibility,
see e.g.,~\cite{bw1,bw2} for the theory and~\cite{bw3} for the \emph{facial
reduction algorithm}. A survey with various views of \FR is
given in~\cite{DrusWolk:16}.
Facial reduction means that there exists a full column rank matrix
$V\in \R^{n\times r}, r<n$, and the corresponding adjoint of the
linear transformation $\VV:\Sn \to \Ss^r$ given in
\index{facial reduction}
\index{substitution!facial reduction}
\[
\textdef{$\VV^*(R)  = VRV^T$},\, R \in \Ss^r,
\]
such that the \emph{substitution} $X=\VV^*(R)$ results in the
\emph{equivalent, regularized, smaller dimensional}, problem
	\begin{equation}
\label{sdp_facialone}
	p^*_{\SDP}=
\min \{ \langle V^{T}CV,R \rangle \;|\; \langle V^{T}A_{i}V,R \rangle  =
b_{i},  \;\; i \in \II \subseteq \{1,\ldots,m\},\;\;  R
\in \Ss^r_+\}.
\footnote{\FR generally results in the
constraints becoming linearly dependent. Therefore, a linearly
independent subset need only be used~\cite{Sremac:2019}.}
\end{equation}
Strict feasibility holds for~\eqref{sdp_facialone}.
The cone $V\Ss^r_+V^T$ is the \textdef{minimal face of the \SDPp},
i.e.,~the smallest face of $\Snp$ that contains the feasible set, $\FF_X$.
And
\[
\range(V) = \range(X), \, \forall X \in \relint(\FF_X).
\]
If $U\in \R^{n\times n-r}$ with  $\Range(U)=\Null(V^T)$, then $W:=UU^T$ is an
\textdef{exposing vector} for the minimal face, i.e.,
\[
X \text{  feasible  } \implies WX = 0.
\]
Let \textdef{$\FF_R$} denote the feasible set for~\eqref{sdp_facialone}.
We emphasize the following constant rank result
for the \FR substitution:
\[
R\in \FF_R,  \, \rank(R)=r \iff  X=\VV^*(R) \in \FF_X, \, \rank(X)=r.
\]
\begin{remark}
\label{rem:FRcompare}
For a typical \FR algorithm for finding the minimal face,
at each iteration the dimension is strictly reduced, and at least
one redundant linear constraint can be discarded, i.e.,~we need at
most $\min\{m,n-1\}$ iterations,
e.g.,~\cite{DrusWolk:16,SWW:17},\cite[Theorem 3.5.4]{Sremac:2019}.
\label{pg:discardconstr}

Note that \FR can also be considered in the original space using
rotations. Each step of \FR involves finding an exposing vector $W=UU^T$ to the
minimal face. Without loss of generality, we can assume that the matrix
$Q=\begin{bmatrix} V & U\end{bmatrix}$ is orthogonal.
Then the \FR that reduces the size
of the problem  $X=\VV^*(R) = VRV^T$ can equivalently be considered as a
rotation (orthogonal congruence):
\[
X =
\begin{bmatrix}
V& U
\end{bmatrix}
\begin{bmatrix}
R & 0 \cr 0 & 0
\end{bmatrix}
\begin{bmatrix}
V& U
\end{bmatrix}^T, \quad
\begin{bmatrix}
R & 0 \cr 0 & 0
\end{bmatrix}
=
\begin{bmatrix}
V& U
\end{bmatrix}^T
X
\begin{bmatrix}
V& U
\end{bmatrix},
\]
i.e.,~after this rotation, we can discard zero blocks and reduce the size
of the problem. We note that this can then be compared to the
 \emph{Constrained Set Invariance Conditions} approach
in~\cite{permenter2017reduction}, where a special projection is used to obtain the
reduced problem. In addition, the approach in~\cite{permenter2017reduction}
performs the projections on the primal-dual problem thus maintaining the
original optimal values of both. In contrast, we emphasize the importance of the
primal problem as being the problem of interest. After \FR we have a
regularized \emph{primal} problem~\eqref{sdp_facialone}
with optimal value the same as
that of the original \emph{primal} problem. In addition, the reduced
program has the important property that the dual of the dual is the primal.
\end{remark}

\subsection{Group invariance and symmetry reduction, \SRp}
\label{sec_group}
We now find a \emph{substitution} using the adjoint linear transformation
 $\BBts(x)$ from~\cref{eq:Btildestarx} below, that provides the \SR in block diagonal form.
We first look at the procedure for simplifying an \SDP that is invariant
under the action of a symmetry group. This approach was introduced by Schrijver~\cite{schrijver1979comparison}; see also the survey
\cite{MR2894697}.
The appropriate algebra isomorphism follows from the Artin-Wedderburn theory~\cite{MR1575142}.
 A more general framework is given in the thesis
\cite{permenter2017reduction}.	More details can be found in
e.g.,~\cite{MR2460523,Klerk:10,gatermann2004symmetry,gijswijt2010matrix}.

	\index{$\GG$,  group of permutation matrices}
	\index{group of permutation matrices, $\GG$}
	
	Let $\GG$ be a nontrivial group of permutation matrices of size $n$.
The \textdef{commutant, $A_{\GG}$}, 	(or \textdef{centralizer
ring}) of $\GG$ is defined as the subspace
	\index{$A_{\GG}$, commutant}
	\begin{equation}
\label{eq:commutant}
	A_{\GG} := \{ X \in \R^{n\times n} \;|\; PX=XP, \; \forall P \in \GG\}.
	\end{equation}
	Thus,  $A_{\GG}$ is the set of matrices that are
	self-permutation-congruent for all $P\in \GG$.
	An equivalent definition of the commutant is
	\[
	A_{\GG} := \{X \in \R^{n\times n} \; | \; \RRG(X) = X\},
	\]
	where
	\index{Reynolds operator, $\RRG(X)$}
	\index{$\RRG(X)$, Reynolds operator}
\begin{equation}\label{def:rey}
		\RRG(X) :=\frac{1}{|{\GG}|} \sum_{P\in {\GG}} PXP^T, \,\, X\in
	\R^{n\times n},
\end{equation}
	is called the \emph{Reynolds operator} (or \textdef{group average})
	of ${\GG}$.  The operator $\RRG$ is the orthogonal projection onto
	the commutant.
	The commutant  $A_{\GG}$ is a \textdef{matrix $*$-algebra}, i.e.,~it is a set of matrices that is closed under addition, scalar multiplication,
	matrix multiplication, and taking  transposes. One may obtain a basis for $A_{\GG}$ from the  orbits of the
	action of $\GG$ on ordered pairs of vertices,
	where the   \textdef{orbit} of $(u_{i},u_{j}) \in \{0,1\}^{n} \times
	\{0,1\}^{n}$ under  the action of  $\GG$ is the set $\{ (Pu_{i},Pu_{j})
	\;|\; P \in \GG \}$,  and $u_i\in \R^n$ is the $i$-th unit vector. In what follows, we denote
\index{basis for $A_{\GG}$, $\{ B_1, \ldots, B_d \}$}
\index{$\{ B_1, \ldots, B_d \}$, basis for $A_{\GG}$}
\index{$d$, dimension of basis for $A_{\GG}$}
\index{dimension of basis for $A_{\GG}$, $d$}
\begin{equation}
\label{eq:basisB}
		\text{basis for  } A_{\GG}:
        \{ B_1, \ldots, B_d \},\, B_i \in \{0,1\}^{n\times n},\, \forall i.
\end{equation}
Let $J \cong J_n$ (resp.~$I \cong I_n$)  denote the matrix of all ones (resp.~the identity matrix)  of appropriate size.
The basis \eqref{eq:basisB} forms a so-called
\emph{coherent configuration}.
\index{$J$, matrix of all ones}
\index{matrix of all ones, $J$}
 \begin{definition}[{\textdef{coherent configuration}}]
\label{coherentconfiguration}
  A set of zero-one $n\times n$ matrices $ \{B_1,\ldots, B_d\}$ is called a
coherent configuration of rank $d$ if it satisfies the following properties:
\begin{enumerate}
\item  $\sum_{i \in \mathcal{I} } B_i = I$  for some
$\mathcal{I} \subset \{1,\ldots,d\}$, and $\sum_{i=1}^d B_i = J$;
\item $B_i^\mathrm{T} \in \{B_1,\ldots, B_d\}$ for $i=1,\ldots,d$;
\item
$B_iB_j \in \spanl  \{B_1,\ldots, B_d\}, \, \forall i,j\in\{1,\ldots,d\}$.
\end{enumerate}
\end{definition}

In what follows we obtain that
the Reynolds operator maps the feasible set $\FF_X$ of
\eqref{sdp_standard} \underline{into} itself and keeps the objective
value the same, i.e.,
\[ 
X \in \FF_X \implies
\RRG(X) \in \FF_X \text{ and }
\langle C, \RRG(X) \rangle = \langle C, X \rangle.
\] 
One can restrict optimization of an \SDP problem  to feasible points in a matrix $*$-algebra that
contains the data matrices of that problem, see e.g.,~\cite{MR2067190,MR2831404}.
In particular, the following result is known.
\begin{theorem}[\cite{MR2831404}, Theorem 4]
\label{thmAlg}
Let $A_{\GG}$ denote a matrix $*$-algebra that contains the data
matrices of an \SDP problem as well as the
identity matrix. If the \SDP problem has an optimal solution, then it has an optimal solution in $A_{\GG} $.
\end{theorem}
\begin{remark}
\label{rem:complxDNN}
In~\cite{MR2831404}, the authors consider complex matrix $*$-algebras. However in most applications, including applications in this paper,
the data matrices are real symmetric matrices and  $A_{\GG} $ has a real
basis, see~\Cref{coherentconfiguration}.
Thus, we consider here the special real case.
The authors in~\cite{MR2831404} also prove that if $A_{\GG}$ has a real
basis, and  the \SDP has an optimal solution, then it has a real optimal solution in $A_{\GG}$.
Real matrix  $*$-algebras are also considered in~\cite{MR2685140,{de2012improved},{MR2546331}}.

In addition, \Cref{thmAlg} holds for \DNNp,
i.e.,~we can move any nonnegativity constraints
on $X$ added to the \SDP in \Cref{thmAlg}
to simple nonnegativity constraints on $x$ in~\cref{eq:FxAg},
see e.g.,~\cite[Pg 5]{Klerk:10}.
\end{remark}

Therefore, we may restrict the feasible set of the optimization problem to
its intersection with $A_{\GG}$. In particular, we can use the basis
matrices and assume that
\begin{equation}
\label{eq:FxAg}
X\in \FF_X \cap A_{\GG}  \Leftrightarrow \left[X =
\sum_{i=1}^{d}x_{i}B_{i} =: \textdef{$\BB^*(x)$} \in \FF_X, \text{  for some  }
	x\in \R ^d \right].
\end{equation}
From now on we assume that $\GG$ is such that $A_{\GG}$
contains the data matrices of~\cref{sdp_standard}.

	\begin{example}[Hamming Graphs]
\label{sec:Hamming}
We now present an example of an algebra that we use later in our numerics.

		The \textdef{Hamming graph}
 $H(d,q)$ is the Cartesian product of $d$ copies of the complete graph $K_q$,
		with vertices represented by $d$-tuples of letters from an alphabet of size $q$. The Hamming distance between vertices $u$ and $v$, denoted by $|(u,v)|$, is the number of positions in which $d$-tuples $u$ and $v$ differ.
		
	The matrices
		$$(B_{i})_{u,v} := \begin{cases}
		1 & \text{ if } |(u,v)| = i \\
		0 & \text{ otherwise}
		\end{cases}, \quad  i = 0,\ldots,d $$
		form a basis of the Bose-Mesner algebra of the Hamming
scheme, see~\cite{delsarte:73}.  In particular, $B_{0} = I$ is the identity matrix and $B_1$ is the adjacency matrix of the  Hamming graph $H(d,q)$ of size $q^d \times q^d$.
In  cases, like for the  Bose-Mesner algebra, when one of the basis
elements equals the identity matrix, it is common to set the index
of the corresponding basis element to zero.
The basis matrices $B_{i}$ can be simultaneously diagonalized by the real, orthogonal matrix $Q$ given by
		$$Q_{u,v} = 2^{-\frac{d}{2}}(-1)^{u^{T}v}.$$
		The distinct elements of the matrix $Q^{T}B_{i}Q$ equal $K_{i}(j)$ ($j=0,\ldots,d$) where
		$$K_{i}(j) := \sum_{h=0}^{i} (-1)^{h}(q-1)^{i-h}{j \choose h} {d-j \choose i-h}, \quad  j = 0,\ldots,d, $$
are \emph{Krawtchouk polynomials}. We denote by $\mu_j := {d \choose j}(q-1)^{j}$ the multiplicity of the $j$-th eigenvalue $K_{i}(j)$.
The elements of the character table $P \in \R^{(d+1) \times (d+1)}$ of the Hamming
scheme $H(d,q)$, given in terms of the Krawtchouk polynomials, are
\[
p_{i,j} := K_{i}(j), \, i,j = 0,\ldots,d.
\]
In the later sections, we use the following well-known orthogonality relations on the Krawtchouk polynomial,
see e.g.,~\cite{delsarte:73},
\begin{equation}\label{orth_kraw}
\sum_{j=0}^{d}K_{r}(j)K_{s}(j)  {d \choose  j}(q-1)^{j} = q^{d}
       {d \choose s}(q-1)^{s} \delta_{r,s}, \quad r,s =0,\ldots,d,
		\end{equation}
where $\delta_{r,s}$ is the Kronecker delta function.
\index{$\delta_{r,s}$, Kronecker delta}
\index{Kronecker delta, $\delta_{r,s}$}
	\end{example}	

\subsubsection{First symmetry reduction using $X=\BB^*(x)$}
\label{sect:firstsymmX}
We now obtain our first reduced program using the substitution
$X=\BB^*(x)$. Note that the program is reduced in the sense that the
feasible set can be smaller though the optimal value remains the same.
\index{substitution!symmetry reduction}
	\begin{equation}
	\label{sdp_afterB}
	p^*_{\SDP}=
	\min \{ \langle \BB(C),x \rangle \;|\;
(\A\circ \BBs)(x)=b, \;\;  \BBs(x) \succeq 0\}, \quad
(\text{substitution  } X=\BB^*(x)).
	\end{equation}
	Here, $\BB$ is the adjoint of $\BB^*$.
In the case of a \DNN relaxation,
the structure of the basis in~\cref{eq:basisB}
allows us to \label{pg:splitting}
equate $X=\BB^*(x) \geq 0$ with the simpler $x\geq 0$.
This changes the standard doubly nonnegative cone into a splitting, the Cartesian product of
the cones $x\geq 0, \BBs(x)\succeq 0$, see~\Cref{rem:DNNPovhRendl,rem:complxDNN}.

A matrix $*$-algebra $\mathcal{M}$ is called \textdef{basic} if $\mathcal{M} = \{
\oplus_{i=1}^{t} M \;|\; M \in \mathbb{C}^{m \times m} \}$, where
$\oplus$ denotes the direct sum of matrices.
A very important decomposition result for matrix $*$-algebras is the following result due to Wedderburn.
\index{basic matrix $*$-algebra}
\begin{theorem}[\cite{MR1575142}]\label{wedd}
Let $\mathcal{M}$ be a matrix $*$-algebra  containing  the identity matrix.
Then there exists a unitary matrix $Q$ such that $Q^{*}\mathcal{M}Q$ is a direct sum of basic  matrix $*$-algebras.
\end{theorem}

The above result is derived for a complex  matrix $*$-algebras. In~\cite{MR2685140},
the authors study   numerical algorithms for block-diagonalization of matrix $*$-algebras over  $\mathbb{R}$.
Unfortunately, the Wedderburn decomposition described in the above
theorem does not directly apply for  $*$-algebras over  reals.
To demonstrate our approach in the section on numerical results we use known orthogonal matrices or a simple heuristics to obtain them.

To simplify our presentation, the matrix $Q$ in~\Cref{wedd} is assumed to be real orthogonal.
(The case when $Q$ is complex can be derived analogously.)
Then, the matrices in the basis $B_{j},\, j= 1,\ldots,d$, can be mutually
block-diagonalized by some orthogonal matrix $Q$.
More precisely, there exists an orthogonal matrix $Q$ such that
we get the following block-diagonal transformation on $B_j$:
\index{$t$, number of blocks}
\index{number of blocks, $t$}
\index{$\Blkdiag$, block diagonal}
\index{block diagonal, $\Blkdiag$}
	\begin{equation}
\label{eq:Qblkdiag}
	\textdef{$\tilde B _j := Q^{T}B_{j}Q =:
    \Blkdiag( (\tilde B_j^k)_{k=1}^t)$},
	\forall j=1,\ldots,d.
	\end{equation}
	For $Q^{T}XQ =  \sum_{j=1}^{d}x_{j}\tilde{B}_{j}$, we now define the
linear transformation for obtaining the block matrix diagonal form:
\begin{equation}
\label{eq:Btildestarx}
	\textdef{$\BBts(x)$} :=  \sum_{j=1}^{d}x_{j}\tilde{B}_{j} = \begin{bmatrix}
	\BBts_{1}(x) &  & \\
	& \ddots & \\
	& & \BBts_{t}(x)
	\end{bmatrix}
=: \Blkdiag((\BBts_k(x))_{k=1}^t),
\end{equation}
	where
\[
\textdef{$\BBts_{k}$}(x)  =: \sum_{j=1}^d x_j\textdef{$\tilde B^k_j$}
 \in \mathcal{S}^{n_{i}}_+
\]
 is the $k$-th
diagonal block of $\BBts(x)$, and the sum of the $t$ block sizes
 $n_{1}+ \ldots +n_{t} = n$.
Thus, for any feasible $X$ we get
\[
X = \BBs(x) = Q\BBts(x)Q^T \in \mathcal{F}_{X}.
\]
\subsubsection{Second symmetry reduction to block diagonal form
using $X=Q\BBts(x)Q^T$}
\label{sect:secsymmred}
We now derive the second reduced program using the substitution
$X=Q\BBts(x)Q^T$. The program is further reduced since we obtain the block
diagonal problem
\index{substitution!block diagonal symmetry reduction}
	\begin{equation}
	\label{sdp_afterBts}
	p^*_{\SDP}=
	\min \{ \langle \BBt(\tilde C),x \rangle \;|\;
(\At\circ \BBts)(x)=b, \;\;  \BBts(x) \succeq 0\},
	\end{equation}
where $\tilde C = Q^TCQ$ and $\At$ is the linear transformation obtained
from $\A$ as follows: $\tilde A_j = Q^TA_jQ, \forall j$.
We denote the corresponding blocks as \textdef{$\tilde A_j^k$}$, \forall j =1,\ldots,d, \forall k = 1,\ldots, t$.
\label{pg:congruence}

We see that the objective in~\cref{sdp_afterBts} satisfies
\[
\textdef{$\tilde c : = \BBt(\tilde C)$} =
(\langle \tilde B_j, \tilde C \rangle) =
(\langle B_j, C \rangle)\in \R^d.
\]
While the $i$-th row of the linear equality constraint
in~\cref{sdp_afterBts}, \textdef{$\tilde A x = b$}, is
\[
\begin{array}{rcl}
b_i
&=&
 (\tilde Ax)_i
\\&:=&
 ((\At\circ \BBts)(x))_i
\\&=&
\langle \tilde A_i, \BBts(x) \rangle
\\&=&
\langle \BBt(\tilde A_i), x \rangle.
\end{array}
\]
Therefore
\begin{equation}
\label{eq:matrixrepresA}
\tilde A_{ij} = (\BBt(\tilde A_i))_j
= \langle \tilde B_j,\tilde A_i\rangle
= \langle B_j,A_i\rangle,
	 \quad i=1, \ldots ,m, \, j=1,\ldots ,d.
\end{equation}
Without loss of generality, we can now define
\[
c:=\tilde c, \quad A := \tilde A.
\]
Moreover, just as for \FRp, the \SR step can
result in $A$ not being full row rank (onto). We then have to choose a
nice (well conditioned) submatrix that is full row rank and use the
resulting subsystem of $Ax=b$. We see below how to do this and
simultaneously obtain strict feasibility.
	
	We can now rewrite the \SDP~\cref{sdp_standard} as
	\begin{equation}
\label{sdp_sys_reduced}
	p^*_{\SDP}=
	\min \{ c^{T}x \;|\; \textdef{$Ax=b$},
	\;\;  \BBts_{k}(x) \succeq 0, \,  k = 1,\ldots,t \}.
	\end{equation}
	For many applications, there are \textdef{repeated blocks}. We
	then take advantage of this to reduce the size of the problem
	and maintain stability.

	The program~\cref{sdp_sys_reduced} is a
	\textdef{symmetry reduced formulation} of~\cref{sdp_standard}. We
	denote its feasible set and feasible slacks as
	\index{$\SSx$, feasible set with $\BBts(x)$}
	\index{feasible set with $\BBts(x)$, $\SSx$}
	\index{$\FF_x$, feasible set with $x$}
	\index{feasible set with $x$, $\FF_x$}
	\begin{equation}
\label{eq:feasBBs}
	\FF_x := \{x \;|\; \BBts(x)\succeq 0,\,
	          Ax=b,\,x\in \R^d\},\quad
	\SSx := \{\BBts(x)\succeq 0\;|\; Ax=b,\,x\in \R^d\}.
	\end{equation}
We denote the \emph{feasibility problem}
for this formulation with data $\BBts,A,b, \Snp$
of the feasible set $\FF_x$ as $\PP_{F_x}$.
	We bear in mind that $\BBts(x)$ is a block-diagonal matrix.
But it is written as a single matrix for convenience in order to
describe \FR for the symmetry reduced program below.

	\index{$\PP_{F_x}$, feasible problem}
	\index{feasible problem$, \PP_{F_x}$}

	Since $\tilde{B}_{1},\ldots,\tilde{B}_{d}$ are block diagonal,
	symmetric matrices, the symmetry reduced formulation is typically much
	smaller than the original problem,
	i.e.,
	\[
	x \in  \Rd, \quad d \ll \sum_{i=1}^dt(n_i) \ll t(n),
	\]
	where $t(k)=k(k+1)/2$ is the triangular number.
	\index{$t(k)=k(k+1)/2$, triangular number}
	\index{triangular number, $t(k)=k(k+1)/2$}

	\section{Facial reduction for the symmetry reduced program}
	\label{sect:FRforsymm}
	In this section, we show how to apply \FR to the
	symmetry reduced  \SDP~\cref{sdp_sys_reduced}.
The key is using the exposing vector view of facial
reduction,~\cite{DrusWolk:16}.
Formally speaking, if an exposing vector of the \textdef{minimal face}
of the \SDP~\cref{sdp_standard} is given, then we are able to construct
a corresponding exposing vector of the minimal face of the symmetry
reduced program~\cref{sdp_sys_reduced}.
In fact, we show that all the exposing vectors of the symmetry reduced program
can be obtained from the exposing vectors of the original program.
In general, one can find exposing vectors from
the original program by exploiting the structure. However, this is lost
after the \SR and results in a more difficult task in finding an
exposing vector.

In addition, we follow the above theme on simply adding on the nonnegativities and
extend many of the results to the \DNN program.
We include results on the singularity degree to emphasize the importance
of \FR for stability and that \SR does \emph{not} increase the singularity degree.

\subsection{Rank preserving}
\label{sect:rankpres}
We begin with showing the maximum rank preserving properties of \SRp. Note that
\[
\begin{array}{rcl}
\max \{\rank(X) \,|\, X \in \FF_X  \}
&=&
\rank(X), \,\, \forall X \in \relint(\FF_X)
\\&=&
\rank(X), \,\, \forall X \in \relint(\face(\FF_X)),
\end{array}
\]
where $\face(\FF_X)$ is the minimal face of $\Snp$ containing the feasible set.
We let $F \unlhd K$ denote that $F$ is a face of the cone $K$.
\index{$\face(S)$, minimal face of $\Snp$ containing $S$}
\index{minimal face of $\Snp$ containing $S$, $\face(S)$}
	\begin{theorem}\label{thm:rank}
Let $r = \max \{ \rank(X) \,|\, X\in \FF_X \}.$ Then
\[
\begin{array}{rcl}
r&=&
\max\left\{\rank\left( {\frac{1}{|\GG|}\sum_{P \in \GG}P^{T}XP} \right)
 \, | \, X \in \FF_X \right\}
\quad \left(=\max \{ \rank(\RRG(X)) \,|\, X \in \FF_X \}\right)
\\&=&
  \max \{ \rank(X) \,|\, X\in \FF_X \cap A_{\GG} \}
\\&=&
  \max \{ \rank(\BBts(x)) \,|\, \BBts(x) \in \SSx \}.
\end{array}
\]

	\end{theorem}
	\begin{proof}
	Let $X \in \FF_X$ be the matrix with maximum rank $r$.
	Then $X$ is in the relative interior of the minimal face $f\unlhd \Snp$ containing $\mathcal{F}_{X}$, i.e.,
	\[
	X\in \relint(f) =
	\begin{bmatrix}
	  V & U
	    \end{bmatrix}
	\begin{bmatrix}
	\Ss^r_{++} & 0\cr
	0  & 0
	    \end{bmatrix}
	\begin{bmatrix}
	  V & U
	    \end{bmatrix}^{\Tr},  \, \text{ for some orthogonal  }
\begin{bmatrix} V & U     \end{bmatrix}.
	\]
		
	The nonsingular congruence $P^{T}XP$ is feasible for each $P \in \GG$,
	and also has rank $r$. Note that
\[
 A,B \in \Snp \implies \rank(A+B) \geq \max\{ \rank(A), \rank(B)\}.
\]
Therefore, applying the Reynolds operator, we have
	\[
	X_{0} = {\frac{1}{|\GG|}\sum_{P \in \GG}P^{T}XP}
	\in \relint(f).
	\]
	Since $X_{0} \in {\A} _{\GG}$, we have
$Q^{T}X_{0}Q \in \SSx \,(= Q^T(\FF_X \cap A_{\GG})Q)$
and it has rank $r$, where $Q$ is the
orthogonal matrix given above in~\cref{eq:Qblkdiag}.
	
	Conversely, if $\BBts(x) \in \SSx$ with rank $r$, then $X := Q
\BBts(x) Q^{T}$ is in $\mathcal{F}_{X}$ with rank $r$.
	\end{proof}
Note that in the proof of~\Cref{thm:rank} we exploit the following known properties of the Reynolds operator:
	$\rank(\RRG(X)) \geq \rank(X),$ which  is  valid for all $X$ that are positive semidefinite, and $\RRG(F_X) =\F_X\cap A_\GG$.

	\begin{corollary}\label{cor:rank}
	The program~\cref{sdp_standard} is strictly feasible if, and
only if, its symmetry reduced program~\cref{sdp_sys_reduced} is strictly feasible.
	\end{corollary}

\begin{remark}
From the proof of~\Cref{thm:rank},
if there is a linear transformation $X = \LL(x)$
with a full rank feasible $\hat X \in \range(\LL), \hat X = \LL(\hat x)$,
then in general we can conclude that the substitution $X=\LL(x)$
results in a \emph{smaller} \SDP with strict feasibility holding at $\hat x$,
i.e.,
\[
\hat X\succ 0, \A(\hat X)=b, \hat X = \LL(\hat x) \implies
\LL(\hat x)\succ 0, (\A\circ\LL)(\hat x) = b.
\]
\end{remark}

\subsection{Exposing vectors}
\label{sect:FRandexp}
For many given combinatorial problems, the semidefinite
relaxation is not strictly feasible, i.e.,~it is degenerate,
ill-posed, and we can apply
\FR~\cite{DrusWolk:16,Sremac:2019,permenter2017reduction}.
From \Cref{sect:rankpres} above, we see that this implies that the symmetry reduced
problem is degenerate as well. Although both \SR and \FR can be performed separately to obtain two independent problems,
there has not been any study  that implements these techniques simultaneously
and efficiently  for  \SDPs, i.e.,~to obtain a symmetry
reduced problem  that also guarantees strict feasibility.
Recall that \cite{Lofberg} combines  partial  \FR and  \SR  for solving {\bf SOS} programs.
	
	In what follows, we show that the exposing vectors of the
symmetry reduced program~\cref{sdp_sys_reduced} can be obtained from the
exposing vectors of the original program~\cref{sdp_standard}.
This enables us to facially reduce the symmetry
reduced program~\cref{sdp_sys_reduced} using the structure from the original problem.
	
	Let $W = UU^{T}$, with $U\in  \R^{n\times (n-r)}$ full column rank;
and let $W$  be  a nonzero exposing vector of a face of $\Snp$ containing the feasible region $\FF_X$ of~\eqref{sdp_standard}.
Let $V \in \R^{n \times r}$ be such that
\index{$n$, order of \SDP matrices}
\index{order of \SDP matrices, $n$}
\index{$m$, number of constraints in \SDPp}
\index{number of constraints in \SDPp, $m$}
\[
\Range(V) = \Null(U^T).
\]
Then \FR means that we can use the substitution
$X=\VV^*(R)=VRV^T$ and obtain the following equivalent, smaller,
formulation of~\eqref{sdp_standard}:
	\begin{equation}\label{sdp_facial}
	 p^*_{\SDP} = \min \{ \langle V^{T}CV,R \rangle \;|\; \langle V^{T}A_{i}V,R \rangle  = b_{i},  \;\;
i \in \II \subseteq \{1,\ldots,m\},
\;\;  R \in \Ss^r_+\}.
	\end{equation}
If $V$ exposes the minimal face containing $\FF_X$, then strict
feasibility holds.
In fact, $\hat R$ strictly feasible corresponds to $\hat X=\VV^*(\hat R) \in \relint(\FF_X)$.
	
The following results show how to find  an exposing vector that
is in the \underline{commutant $A_{\GG}$}.
\index{$F \unlhd K$, face}
\index{face, $F \unlhd K$}
\begin{lemma} \label{lemma:exposing_vectorinAG}
Let $W$ be an exposing vector of rank $d$ of a face  $\FF \unlhd \Snp, \FF_X\subseteq \FF$.
Then there exists an exposing vector $W_{\GG} \in A_{\GG}$ of $\FF$
with $\rank(W_{\GG}) \geq d$.
	\end{lemma}
	\begin{proof}
Let $W$ be the exposing vector of rank $d$, i.e., ~$W\succeq 0$ and
\[
X \in \FF_X \implies \langle W , X \rangle = 0.
\]
	Since~\cref{sdp_standard} is  $\GG$-invariant,
$PXP^{T} \in \FF_X$ for every $P \in \GG$, we conclude that
	$$\langle W, PXP^{T} \rangle = \langle P^{T}WP, X \rangle =  0.$$
	Therefore, $P^{T}WP \succeq 0$ is an exposing vector of rank $d$.
	Thus $W_{\GG} = \frac{1}{|\GG|} \sum_{P \in \GG} P^{T}WP$ is an exposing vector of $\FF$.

That the rank is at least $d$ follows from taking the sum of nonsingular congruences of $W\succeq~0$.
	\end{proof}
	
\Cref{lemma:exposing_vectorinAG} shows that $A_{\GG}$ contains exposing vectors.
This result is a valuable addition to the list of objects that exhibit
symmetry, see for example:
dual solutions and the central path in~\cite{KoKoHa:94};
solutions on the central path and some search directions of primal-dual
interior-point methods, in~\cite{KannoOhsakiMurotaKath};
and infeasibility certificates, in~\cite{ParriloPermenterP20}.

Note that one can obtain an exposing vector $W_{\GG} \in 	A_{\GG}$ from an exposing vector $W$ by using
the Reynolds operator, as done in~\Cref{lemma:exposing_vectorinAG}.
However, in some cases $W_{\GG}$ can be more easily derived, as our
examples in the later numerical sections show.
We now continue and show that $Q^{T}W_{\GG} Q$ is also an exposing vector.
	
\begin{lemma}\label{prop:exposeG1}
Let  $W \in A_{\GG}$ be an exposing vector of a face
$\FF\unlhd \Snp, \FF_X\subseteq \FF$,
Let $Q$ be the orthogonal matrix given above in~\cref{eq:Qblkdiag}.
	Then $\widetilde{W} = Q^{T}WQ$ exposes  a face of $\Snp$ containing $\SSx$.		
	\end{lemma}
	
	\begin{proof}
\index{$\blkdiag$, adjoint of $\Blkdiag$}
\index{adjoint of $\Blkdiag$, $\blkdiag$}
	Let	
		\[
		Z =\sum_{i=1}^{d} x_i \tilde B_i = Q^{T}\left( \sum_{i=1}^{d} x_i B_i\right)Q
		\in \SSx.
		\]
		Then, by construction $Z$ is a block-diagonal matrix,
		say $Z = \Blkdiag(Z_{1},\ldots,Z_{t})$.
	Now, since $W$ is an exposing vector of the face of $\Snp$ containing  $\FF_X$ we have
		\[
		\begin{array}{rcll}
		WX = 0, ~\forall X \in \FF_X
		&\implies &
		WX = 0, & \forall X =\sum\limits_i x_i B_i\succeq 0,
       \mbox{   for some   } x \text{   with   } Ax=b 	\\[1ex]
	&\implies &
		\widetilde{W}Z = 0, & \forall Z \in \SSx,
		\\
		\end{array}
		\]
		where $\widetilde{W} = Q^{T}WQ \succeq 0$. Thus, $\widetilde{W}$ is an
	exposing vector of a proper face of $\Snp$ containing $\SSx$.
		
		Since $Z = \Blkdiag(Z_{1},\ldots,Z_{t})$ is a block-diagonal
	matrix and $W \in A_{\GG}$, we have that
	  $\widetilde{W} = \Blkdiag(\widetilde{W}_{1},\ldots,\widetilde{W}_{t})$ with $\widetilde{W}_{i}$ the corresponding $i$-th diagonal block of $Q^{T}WQ$.
	\end{proof}
	
	Since we may assume $W \in A_{\GG}$, the exposing vector $Q^{T}WQ$ is a block-diagonal matrix.
	Now, let us show that $Q^{T}WQ$ exposes the minimal face of $\Snp$
	containing $\SSx$, $\face(\SSx)$.  It suffices to show that the rank of
	$Q^{T}WQ$  is $n-r$, see~\Cref{thm:rank}.
	
	\begin{theorem}
\label{thm:exposeG2}
	Let $W\in A_{\GG}$ be an exposing vector of $\face(\FF_X)$,
the minimal face of $\Snp$ containing $\FF_X$.
Then the block-diagonal matrix $\widetilde{W} = Q^{T}WQ$ exposes $\face(\SSx)$,
the minimal face of $\Snp$ containing  $\SSx$.
	\end{theorem}
	\begin{proof}
	The minimality follows from~\Cref{thm:rank}, as $\rank(\widetilde{W})=\rank(W) = n-r$.
	\end{proof}
	
	Now let  $\widetilde{W} = Q^{T}WQ$ expose the minimal
face of $\Snp$ containing  $\SSx$, and let
\[
\widetilde{W} =\Blkdiag(\widetilde{W}_{1},\ldots,\widetilde{W}_{t}),
\quad \widetilde{W}_{i}=\tilde{U}_i\tilde{U}_i^T, \,\tilde U_i \text{
full rank, } \,i=1,\ldots,t.
\]
Let  $\tilde{V}_i$  be a full rank matrix whose columns form a basis for the
orthogonal complement to the columns of $\tilde{U}_{i}, i=1,\ldots,t$.
Take $\tilde{V} = \Blkdiag(\tilde{V}_{1},\ldots,\tilde{V}_{t})$.
Then, the facially reduced formulation of \cref{sdp_sys_reduced} is
	\index{\FR optimal value, $p^*_{\FR}$}
	\index{$p^*_{\FR}$, \FR optimal value}
\begin{align}
\label{sdp_sys_facial}
\begin{split}
p^*_{FR} = & \min \{ c^{T}x \;|\; \textdef{$Ax=b$},
\;\;   \BBts(x) = \tilde{V}\tilde{R}\tilde{V}^{T}, \; \tilde{R} \succeq 0\}\\
= & \min \{ c^{T}x \;|\; \textdef{$Ax=b$},
\;\;  \BBts_{k}(x) = \tilde{V}_{k}\tilde{R}_{k}\tilde{V}_{k}^{T}, \; \tilde{R}_{k} \succeq 0, \; \forall k = 1,\ldots,t\},
\end{split}
\end{align}
where $\tilde{V}_{k}\tilde{R}_{k}\tilde{V}_{k}^{T}$ is the corresponding $k$-th block of $\BBts(x)$,
and $\tilde{R} =  \Blkdiag(\tilde{R}_{1},\ldots,\tilde{R}_{t})$.
Note that some of the blocks $\BBts_{k}(x)$,  and  corresponding $\tilde{R}_{k}$, might be the same and thus can be
removed in the computation, see~\Cref{wedd}.

\begin{remark}\label{rem:steps}  We have assumed that an exposing vector of
the minimum face of the original \SDP \cref{sdp_standard} is available.
If this is not the case,
then we can find a strictly feasible formulation of \cref{sdp_standard},
and an exposing vector of the minimum face for the original problem,
by using a finite number (at most  $\min\{m,n-1\}$) facial reduction
steps, e.g.,~\cite{DrusWolk:16,SWW:17,Sremac:2019}.

We note here that  reduction techniques based on the Constrained Set Invariance Conditions, such as $*$-algebra techniques,
can obtain strict feasibility by removing zero blocks after the appropriate
projection, see~\cite{permenter2017reduction}.
\end{remark}

\subsubsection{Order of reductions} \label{sect:reductionOrder}
To obtain the combined symmetry and facially reduced semidefinite
program~\cref{sdp_sys_facial},
we first apply \SR to~\cref{sdp_standard}, and then follow this with
\FR  to the form in \cref{sdp_sys_reduced}.
A natural question is whether the order of reduction matters.

Note that the objective $\langle V^{T}CV, R \rangle$ and the constraints
$\langle V^{T}A_{i}V,R \rangle  = b_{i},  \;  i \in \II \subseteq
\{1,\ldots,m\}$, of the facially reduced program \eqref{sdp_facialone}, see also \eqref{sdp_facial},
depend on the choice of $V$.
We now show that the choice of this $V$ is crucial when reversing the
order of reductions \FR and \SRp.
For a naive choice of $V$, we can lose all symmetries structure
for \SR in~\Cref{sec_group}.
For example, assume the data matrices $C,A_{1},\ldots,A_{m}$ of the
original problem are invariant under a non-trivial permutation group
$\GG$, i.e.,~they are in the commutant $\A_\GG$,
see~\Cref{eq:commutant}.
However the data matrices $V^{T}CV,V^{T}A_{1}V,\ldots,V^{T}A_{m}V$ of the
facially reduced problem may not be invariant under any non-trivial
group of permutation matrices for the given $V$.
Note that we can always replace $V \leftarrow VS$ using any invertible
$S$. Then an arbitrary invertible congruence $S^TV^TA_iVS$ will destroy the
symmetry structure in the constraint matrices.

\begin{lemma}\label{ro_range}
Let $V, \tilde{V}, Q$ be given as in the above paragraph, and
in~\Cref{thm:exposeG2} and \cref{sdp_sys_facial}. Then
\[
\range(V) =  \range(Q\tilde{V}).
\]
\end{lemma}
\begin{proof}
\[
\begin{array}{rcl}
\range(\tilde{V}) = \nul(Q^{T}WQ)
  &\implies &
Q(\range(\tilde{V})) = Q(\nul(Q^{T}WQ))
  \\&\implies &
\range(Q\tilde{V}) = \nul(Q^{T}WQQ^T)
  \\&\implies &
\range(Q\tilde{V}) = \nul(W).
\end{array}
\]
\end{proof}

From \Cref{ro_range}, we can set $V = Q\tilde{V}$ for \FRp.
The objective and constraints become
$$
\langle V^{T}CV ,R \rangle =
 \langle \tilde{V}^{T}\tilde{C}\tilde{V} ,R \rangle, \quad
\langle V^{T}A_{i}V ,R \rangle =  \langle
\tilde{V}^{T}\tilde{A}_{i}\tilde{V},R \rangle = b_i, \forall i.
$$
As $\tilde{C},\tilde{A}_{i}$ and $\tilde{V}$ are block-diagonal matrices with appropriate sizes, the data matrices $\tilde{V}^{T}\tilde{C}\tilde{V}$ and $\tilde{V}^{T}\tilde{A}_{i}\tilde{V}$ are block-diagonal as well.
Since $R$ is also a block-diagonal matrix, this choice of $V$ implicitly  exploits symmetry of the original problem data.
The reduction in this case is a special case of a general reduction
technique  known as a projection-based method,
see~\cite{permenter2017reduction} and \Cref{rem:FRcompare} above.

We conclude that if \FR is implemented first, then for \SR to follow
it is crucial to find a suitable matrix $V$ to retain the
symmetry structure in the facially reduced problem.
Therefore, it is more convenient for our approach to apply symmetry reduction before
facial reduction and exploit the simple relation with the exposing vectors.
However, for some other symmetry  reduction methods it might be more appropriate to
do first \FR and then \SR, assuming that some symmetry is preserved after \FR and that the \SR and \FR procedures have comparable costs, see e.g., \cite{Lofberg}.

\subsection{Doubly nonnegative, \DNNp, program}
\label{sect:newDNN}
In this section, the theory above for \SDP is extended to doubly nonnegative program.
We show that if a maximal exposing vector $W$ (see \Cref{def:max}) of the \DNN program \cref{dnn_standard}  is given, then we can construct an exposing vector for the minimal face of the symmetry reduced \DNN program \cref{dnn_sys_reduced}.
This results in a strictly feasible symmetry reduced \DNN program \cref{dnn_sys_facial3}.

Note that in addition to positive definiteness, we need $X > 0$,
elementwise positivity, for strict feasibility to hold for the \DNN relaxation.
We denote the cone of nonnegative symmetric matrices of order $n$ by $\Nn$.
The following extends \cite[Proposition 2.3]{MR3928442}
for the intersection of faces to include exposing vectors.
\index{$Nn$, symmetric nonnegative}
\index{symmetric nonnegative, $\Nn$}
\begin{theorem}
\label{thm:expDNN}
Let $F_S \unlhd \Snp$, and let $F_N \unlhd \Nn$.
 Let $W_S\in \Snp, W_N\in \Nn$ be
exposing vectors for $F_S,F_N$, respectively. Then
\[
W_S+W_N \text{ is an exposing vector for } F_S\cap F_N \text{ for }
\DNN^n.
\]
\end{theorem}
\begin{proof}
Note that since $\Nn$ is a polyhedral cone, and both $\Snp,\Nn$ are
self-dual, we get that the \textdef{dual cone} (nonnegative polar)
\[
(\DNN^n)^* = \Snp + \Nn.
\]
Note that, by abuse of notation,
\[
\langle W_S,F_S\rangle = 0, \quad  \langle W_N,F_N\rangle = 0.
\]
We can now take the sum of the exposing vectors and it is clearly an
exposing vector on the intersection of the faces.
\end{proof}
\begin{remark}
\label{rem:selfreprod}
\begin{enumerate}
\item
\label{item:remselfreprod}
\Cref{thm:expDNN}   holds for arbitrary closed convex cones.

\item
For our application, we note that the intersection $F_S\cap F_N$ is
characterized by the facial representation $X\in V\Ss^r_+ V^T$ and
$X_{ij}=0$ for appropriate indices $i,\!j$.
\FR on the \PSD cone allows one to obtain a new \SDP problem of lower
dimension, since  any face of the \PSD cone is  isomorphic to a smaller \PSD cone.
However, the \DNN cone  does not possess this property. Namely, a face of a \DNN cone is not necessarily
isomorphic to a smaller \DNN cone. However, our splitting still allows for a simplification based
on the Cartesian product of a \SDP cone and a $\Nn$ cone,
see~\cref{sdp_afterB}, the paragraph after~\cref{sdp_afterB},  and \Cref{rem:DNNPovhRendl}, below.

\end{enumerate}
\end{remark}

The \DNN program is defined as
	\index{\DNN optimal value, $p^*_{\DNN}$}
	\index{$p^*_{\DNN}$, \DNN optimal value}
\begin{equation}
\label{dnn_standard}
(\textdef{$\PP_\DNN$}) \qquad p^*_{\DNN} :=	\min \{ \langle C,X \rangle \;|\; \textdef{$\A(X)$}=b, \,  X \in \DNN^n\}.
\end{equation}
The symmetry reduced formulation of the \DNN program \cref{dnn_standard} is
\begin{equation}
\label{dnn_sys_reduced}
p^*_{\DNN}=
\min \{ c^{T}x \;|\; \textdef{$Ax=b$},
\, x \geq 0,\  \BBts_{k}(x) \succeq 0, \,  k = 1,\ldots,t \},
\end{equation}
see \cref{eq:Btildestarx}  for the definition of $\BBts_{k}(x)$.
 Recall that the symmetry reduced formulation of the \SDP program \cref{sdp_standard} is \cref{sdp_sys_reduced}.
The ambient  cone of the symmetry reduced program
\cref{dnn_sys_reduced} is the
Cartesian product of cones $(\R^{d}_{+},\Ss_+^{n_1},\ldots,\Ss_+^{n_t})$.

Let $W \in \DNN^{*}$ be an  exposing vector of \cref{dnn_standard}.  Then $W = W_{S} + W_{N}$ for some $W_{S} \in \Snp$ and $W_{N} \in \Nn$. The exposing vector $W \in \DNN^{*}$  satisfies $\langle W , X \rangle = 0$ for every feasible $X$ of \cref{dnn_standard}. Since it also holds that $\langle W_{S}, X\rangle \geq 0$ and $\langle W_{N} , X \rangle \geq 0$, we have $$\langle W_{S}, X\rangle = \langle W_{N} , X \rangle = 0,$$ for every feasible $X$ of \cref{dnn_standard}.

We are going to construct an exposing vector for the symmetry reduced
program \cref{dnn_sys_reduced} by using $W$. Here the exposing vectors
$(\widetilde{W_{n_{1}}},\ldots,\widetilde{W_{n_{t}}})$ for the
semidefinite cones $(\Ss^{n_1},\ldots,\Ss^{n_t})$ can be derived in the same way as before.
Therefore we only have to find an exposing vector for the nonnegative cone $\R^{d}_+$.
Let $x$ be feasible for  \cref{dnn_sys_reduced}.
Then $X = \BB^{*}(x)$ is feasible for \cref{dnn_standard}. We have
$$\langle W_{N} , X \rangle = \langle W_{N} , \BB^{*}(x) \rangle = \langle \BB(W_{N}), x \rangle = 0.$$
Define $w := \BB(W_{N})$. Since $W_{N}$ is nonnegative and $\left(\BB(W_{N})\right)_{i} = \langle B_{i}, W_{N} \rangle$ for some zero-one matrix $B_{i}$, the vector $w$ is nonnegative. Then $\langle w,x\rangle = 0$ implies that $w$ is an exposing vector for the cone $\R^{d}_+$ of  \cref{dnn_sys_reduced}.

Thus facial reduction for the nonnegative cone $\R_{+}^{d}$ simply
removes the entries $x_{i}$ associated to positive entries $w_{i} > 0$
from the program. Let $\bar{x}$ be the vector obtained by removing
these entries from $x$. Define the new data matrices
$\bar{c},\bar{A}$ correspondingly. The facial reduction for the
semidefinite cones are handled in the same way as before. This yields the following facially reduced formulation of
\cref{dnn_sys_reduced}
	\index{\FRp,\DNN optimal value, $p^*_{\FRp,\DNN}$}
	\index{$p^*_{\FRp,\DNN}$, \FRp,\DNN optimal value}
\begin{align}
\label{dnn_sys_facial3}
\begin{split}
p^*_{\DNN} = & \min \{ \bar{c}^{T}\bar{x} \;|\; \textdef{$\bar{A}\bar{x}=b$}, \, \bar{x} \geq 0, \,
  \BBts_{k}(\bar{x}) = \tilde{V}_{k}\tilde{R}_{k}\tilde{V}_{k}^{T}, \, \tilde{R}_{k} \succeq 0, \, \forall k = 1,\ldots,t\}.
\end{split}
\end{align}

 We show below that if $W$ is a maximal exposing vector of \cref{dnn_standard}, see \Cref{def:max},
then the facially and symmetry reduced  program \cref{dnn_sys_facial3} is strictly feasible.
In the sequel,  we denote by  $\text{supp}(M)$ the support of the matrix $M$. 

\begin{definition}\label{def:max}
An exposing vector $W$ of the $\DNN$ program \cref{dnn_standard} is maximal if  it has a
	decomposition $W = W_{S} + W_{N}$ for some $W_{S} \in \Snp$ and $W_{N} \in \Nn$ satisfying:
\begin{enumerate}[label=(\roman*)]
	\item $\rank( W_{S})$ is maximal;
	\item
	\label{item:posentries}
	the number of positive entries in $W_{N}$ is maximal,
	i.e., $\text{supp}(W_{N}^\prime) \subseteq \text{supp}(W_{N})$
	for any other exposing vector $W_{N}$.
\end{enumerate}
\end{definition}
 \label{pg:maximalityass}
Note that if $W_{N},W_{N}^\prime \in \Nn$ are exposing vectors for a \DNN program, then $W_{N} + W_{N}^\prime \in
\Nn$ is also an exposing vector.
Therefore the support of $W_N$ in the decomposition of a maximal exposing vector $W$ is unique in \Cref{def:max}.
\label{pg:maximality}

\begin{theorem}\label{thm:DNN_FR}
Assume that $W = W_{S} + W_{N}$, where $W_{S} \in \Snp$, $W_{N} \in \Nn$, is a maximal exposing vector of the $\DNN$ program \cref{dnn_standard}. Then  $W_{S}$ and $W_{N}$ are exposing vectors for the minimal face of the symmetry reduced program \cref{dnn_sys_reduced}, or equivalently, the facially and symmetry reduced program \eqref{dnn_sys_facial3} is strictly feasible.
\end{theorem}

\begin{proof}
Assume, for the sake of contradiction, that \cref{dnn_sys_facial3} is not strictly feasible.
The existence of a feasible solution for  \cref{dnn_sys_facial3} such that $\tilde{R}_{k} \succ 0$ for $k = 1,\ldots,t$ can be derived in the same way as before, see \Cref{thm:exposeG2}.
Therefore,  we consider here only the case that there does not exist feasible $\bar{x}$ for  \cref{dnn_sys_facial3} that is strictly positive.
Then there exists an exposing vector $w^\prime \in
	\R_{+}^{d}$ for \cref{dnn_sys_reduced} such that $\text{supp}(w)
	\subsetneq \text{supp}(w^\prime)$. Let $W_{N}^\prime :=
	\BB^{*}(w^\prime) \in \Nn$. Then $\text{supp}(W_{N}) \subsetneq
	\text{supp}(W_{N}^\prime)$. Let $X \in \DNN$ be feasible for \cref{dnn_standard}. Then $\RRG(X) = \BB^{*}(x) \in \DNN$ for some $x$ feasible for \cref{dnn_sys_reduced}, and thus
	$$\langle W_{N}^\prime ,\RRG(X)  \rangle = \langle w^\prime,x \rangle = 0.$$
	But $\text{supp}(X) \subseteq \text{supp}(\RRG(X))$, this means that
	$\langle W_{N}^\prime , X \rangle = 0$. Thus $W_{N}^\prime$ is an exposing vector
	\label{pg:forrevision}
	for \cref{dnn_standard} such that $\text{supp}(W_{N}) \subsetneq
	\text{supp}(W_{N}^\prime)$. This contradicts the maximality of $W$. Thus
	the program \cref{dnn_sys_facial3} is strictly feasible. Note
	that the fact that we could move the nonnegativity to the reduced
	variable $x$ was essential for obtaining the Slater condition.
\end{proof}

\begin{remark}\label{rgPrt ofProof}
One can also prove  \Cref{thm:DNN_FR} by exploiting the following  properties of $\RRG(X)$, the Reynolds operator of ${\GG}$,  see \cref{def:rey}.
For all feasible $X$, it holds that
		\[
		\rank \left(\RRG(X)\right) \ge  \rank(X),~   \text{supp} (\RRG(X)) \supseteq \text{supp}(X).
		\]
	\end{remark}

\subsubsection{Facial reduction from the ground set for \DNNp}
\label{sect:leventext}
Our applications involve quadratic models of hard combinatorial
problems. We now see that the view of strict feasibility and \FR
in~\cite[Theorems 3.1, 3.2]{Tun:01} can be easily extended from \SDP to
\DNNp.

We follow the notation in~\cite{Tun:01} and define the feasible set
or \textdef{ground set} of a quadratically constrained program as:
$$\QQ := \left \{ x \in \mathbb{R}^{n} \;|\; \mathcal{A}\left(\begin{bmatrix}
1 & x^{T} \\
x & xx^{T}
\end{bmatrix}\right) = 0, \, x\geq 0 \right \},$$
where $\A$ is a linear transformation.
The relaxation, lifting, is then given by
$$
\hat{\QQ} := \left\{ \begin{bmatrix}
1 & x^{T} \\
x & X
\end{bmatrix} \in \DNN^{n+1} \;|\; \mathcal{A}\left(\begin{bmatrix}
1 & x^{T} \\
x & X
\end{bmatrix}\right) = 0
\right\}.$$
Let the \emph{gangster set} for  $\QQ$ be defined as
\[
\GG_\QQ = \left \{(i,j) : x_ix_j = 0, \, \forall x \in \QQ \right \}
\]
with complement $\GG_\QQ^c$.
Let the {gangster set} for  $\hat  \QQ$  be defined as
\[ \label{defGang}
\GG_{\hat \QQ} = \left \{ (i,j) : X_{ij}= 0, \, \mbox{for all} \,  \begin{bmatrix}
1 & x^{T} \\
x & X
\end{bmatrix}  \in {\hat \QQ } \right \}.
\]
Note that here the gangster sets are equal
 $\GG_{\QQ} =\GG_{\hat \QQ}$, with appropriate indices.
However, for a general \DNNp, we are
not given the ground set and the gangster set is defined for the lifted
problem only.

In what follows we use \textdef{$e_k$ or $e$} when the meaning is clear, to denote the vector of all ones  of order $k$.

\begin{theorem}[Slater]
\label{thm:slaterlevent}
Suppose that $\conv(\QQ)$ is full dimensional and that $\GG_\QQ =
\emptyset$.
Then the Slater condition holds for $\hat{\QQ}$.
\end{theorem}
\begin{proof}
By the assumption, we can choose the finite set of vectors
\begin{equation}
\label{eq:pospts}
\left\{v^{ij} \in \QQ \,|\,  v^{i,j}_iv^{i,j}_j > 0, \, \text{ for each }
(i,j)\in \{1,\ldots,n\} \times \{1,\ldots,n\} \right\}.
\end{equation}
As in~\cite{Tun:01}, we choose an affine linear independent set
$\{x_i\}_{i=1}^{n+1} \subseteq \QQ$,
and form the matrix by adding ones and the $v^{i,j}$ defined in~\cref{eq:pospts}:
\[
V :=
\begin{bmatrix}  {e_{n+1}^T} & {e_{n^2}^T} \cr
\begin{bmatrix}  x^1, \ldots, x^{n+1} \end{bmatrix}&
\begin{bmatrix}  v^{1,1}, v^{1,2}, \ldots, v^{n,n} \end{bmatrix}
\end{bmatrix}
\]
We lift and get the Slater point $W:= VV^T\in \hat \QQ, W\succ 0, W>0$.
\end{proof}

We now extend this to obtain \FRp. We use our exposing vector viewpoint
rather than the internal view in~\cite[Theorem 3.2]{Tun:01}.
We note that we can not move here the nonnegativity constraints onto $R$ as
is done for our applications after \SRp. Moreover, though
the Slater condition holds for the \FR feasible set
in~\cref{eq:genSlater}, it is not necessarily true that the
\textdef{Mangasarian-Fromovitz constraint qualification} holds,
since some of the
linear equality constraints typically become redundant after \FRp.
We can however discard redundant equality constraints.
\begin{theorem}[facial reduction]
\label{thm:FRalaTUNforDNN}
Suppose that the affine hull, $\aff(\conv(\QQ))= \LL$ and $\dim(\LL)=d$.
Then there exist $A$ and $b$ with $A$ full row rank such that
\[
\LL = \{x\in \Rn \,|\, Ax=b\}.
\]
Let $U = \begin{bmatrix}-b^T \cr A^T\end{bmatrix}$ and $V$ be full column
rank with $\range(V) = \nul(U)$.
Then there exists a Slater point $\hat R$
for the \FRp, \DNN feasible set
\begin{equation}
\label{eq:genSlater}
\hat \QQ_R  =
       \left\{ R \in \Ss^{d+1} \,|\,
      R\succeq 0, \,(VRV^T)_{\GG_\QQ^c} \geq 0,\,
               (VRV^T)_{\GG_\QQ} = 0,\,
               \A\left(VRV^T\right) = 0
 \right\},
\end{equation}
where  $(VRV^T)_S$  is  the vector with indices chosen from the set $S$,
and
\begin{equation}
      \hat R\succ 0, \,(V\hat RV^T)_{\GG_\QQ^c} > 0,\,
               (V\hat RV^T)_{\GG_\QQ} = 0,\,
               \A\left(V\hat RV^T\right) = 0.
\end{equation}
\end{theorem}

\begin{proof}
The proof is as for~\Cref{thm:slaterlevent} after
observing that $UU^T$ is an exposing vector. More precisely, from
\begin{equation}
\label{eq:liftedgrset}
Ax-b =0
\iff
\begin{bmatrix} 1 \cr x \end{bmatrix}^T
\begin{bmatrix} -b^T \cr A^T \end{bmatrix} =0
\iff
\begin{bmatrix} 1 \cr x \end{bmatrix}
\begin{bmatrix} 1 \cr x \end{bmatrix}^T
\begin{bmatrix} -b^T \cr A^T \end{bmatrix}
\begin{bmatrix} -b^T \cr A^T \end{bmatrix}^T
=0,
\end{equation}
we see that $YUU^T=0$ for all lifted $Y$, and therefore also for the minimal
face.
Thus $UU^T$ is an exposing vector. The result follows after
restricting the
selection in~\cref{eq:pospts} to the complement $\GG_\QQ^c$.
\end{proof}

{\subsection{Singularity degree}
\label{sect:singDeg}

The \emph{singularity degree} defined for the
\emph{semidefinite} feasibility problem
$\PP_F$~\cref{eq:feasprob}, and denoted by $\sd(\PP_F)$,
is the minimum number of steps with a nonzero exposing vector,
for the \FR algorithm to
terminate with the minimal face. For $\PP_\SDP$ this means we
terminate with a strictly feasible problem.
Singularity degree was introduced for
$\PP_\SDP$ in~\cite{S98lmi} to show that \SDP
feasibility problems always admit a
\emph{H\"older error bound},\footnote{\label{ftnt:sdSturm}Our definition of singularity degree does not coincide with the definition from~\cite{S98lmi} when $\PP_F=\{0\}$.
In this case our definition gives $\sd(\PP_F)\geq 1$, while Sturm
defines  $\sd(\PP_F)=0$. See also~\cite{MR3845279} for discussions
on the definition; and
see~\cite{SWW:17} for lower bound results.}
more precisely, let $d=\sd(\PP_F)$, $\LL = \{X \,|\, \A(X) = b \}$, $U\subset \Sn$ be compact.
Let $\dist$ denote the norm-distance to a set. Then it is shown
in~\cite{S98lmi} that there exists $c>0$ such that
\index{$\dist$, norm-distance to a set}
\index{norm-distance to a set, $\dist$}
\[
\dist(X,\FF_X) \leq c\left(\dist^{2^{-d}}(X,\Snp)+
                         \dist^{2^{-d}}(X,\LL) \right), \, \forall X\in U.
\]
Remarkably, the exponent $2^{-d}$ is independent of $m,n$ and
the rank of the matrices in $\FF_X$.
It strongly indicates the importance of \FR for \SDPp,
especially when obtaining approximate solutions with splitting type methods.
This is illustrated by our numerical results,
see~\Cref{table:QueenRez} below,  where lower bounds obtained by \ADMM are dramatically
better than those for \IPMp.

In this section, we show that the singularity degree of a symmetry
reduced program is equal to the singularity degree of the original problem, see~\Cref{thm:sdleq}. Thus, we provide a
heuristic indication that this error measure does not grow when applying \SRp.
Of course, after completing \FRp, the singularity degree is optimal, $0$.

\index{singularity degree, $\sd(\PP_F)$}
\index{$\sd(\PP_F)$, singularity degree}

The facial reduction algorithm applied to the semidefinite program \cref{sdp_standard} is described as follows.
At the $k$-th step, \FR finds an exposing vector of the feasible region of the reduced \SDP of \cref{sdp_standard}
\begin{equation}\label{sd_sdp}
\{ R \in \mathcal{S}^{r_{k}}_{+} \;|\;  \mathcal{A}_{V}(R) = b_\II\},
\text{ with }
\mathcal{A}_{V}(R) = \left(\langle V^{T}A_{i}V ,R \rangle \right)_{i\in
\II} \in \mathbb{R}^{|\II|}, \, \II\subseteq \{1,\ldots,m\}.
\end{equation}
Here $V$ is a given matrix updated after each \FR step.
In the first step,  $V$ is the identity matrix and
 \cref{sd_sdp} is the feasible region $\FF_X$ of the original problem \cref{sdp_standard}.
An exposing vector is then obtained by solving the following
\textdef{auxiliary system} for $y$:
\begin{equation}\label{sdp_aux}
0 \neq \mathcal{A}_{V}^{*}(y) \succeq 0 \text{ and } b^{T}y \leq 0.
\end{equation}
If $y$ exists, then $W = A_{V}^{*}(y) \in
\mathbb{R}^{r_{k} \times r_{k}}$ is an exposing vector of the feasible
region \cref{sd_sdp}. We then do as follows:
\begin{equation}
\begin{array}{rl}
\label{eq:Vprime}
(i) &  \text{compute $V^\prime \in \mathbb{R}^{r_{k} \times r_{k+1}}$, full rank, $\range(V^\prime) =  \nul (W);$} \\
(ii) & \text{set  $V\leftarrow  VV^\prime\in \mathbb{R}^{n \times r_{k+1}}$}; \\
(iii) & \text{repeat from \cref{sdp_aux}}.
\end{array}
\end{equation}

At the $k$-th step, we have computed a vector $y$ and a matrix $V^\prime$
that determines the facially reduced formulation at the next step.
Choosing exposing vectors with maximum possible rank leads to the fewest
iterations,  see e.g.,~\cite{SWW:17,MR3845279}.
For completeness, we now show
that the number of iterations in the facial reduction algorithm only depends
on the choice of $y$ and not on the choice of $V^\prime$.
\begin{lemma}\label{sd_inv}
The total number of facial reduction steps does not depend on the choice
of $V^\prime$ and  $V$ in \cref{eq:Vprime}.
\end{lemma}
\begin{proof}
Assume that $y$ satisfies the auxiliary system \cref{sdp_aux} for the feasible region \cref{sd_sdp}. If we replace $V \in \mathbb{R}^{n \times r_{k}}$ in \cref{sd_sdp}, with $VS \in
\mathbb{R}^{n \times r_{k}}$, for some invertible matrix $S \in
\mathbb{R}^{r_{k} \times r_{k}}$,
then the same vector $y$ satisfies the new auxiliary system, as $b^{T}y \leq 0$ and
$$W_{S} := \mathcal{A}_{VS}^{*}(y) = \sum_{i=1}^{m} (S^{T}V^{T}A_{i}VS)y_{i} = S^{T}\mathcal{A}_{V}^{*}(y)S = S^{T}WS \succeq 0.$$
Since $S$ is invertible, it holds that $W_{S} \neq 0$ and $\rank (W_{S}) = \rank(W)$. Thus, we obtain the same reduction in the problem size at the $k$-th step.

As $\Null(W_{S}) = S^{-1}\Null(W)$, we have $S^{-1}V^\prime \in \Null(W_{S})$, where $V'$ satisfies $\range(V') = \Null(W)$ as in \cref{eq:Vprime}. For any invertible matrix $T\in \mathbb{R}^{r_{(k+1)} \times r_{(k+1)}}$, we have that $V_{S}^\prime = S^{-1}V^\prime T \in \Null(W_{S})$. Thus, in the second step of \cref{eq:Vprime}, we have
$$
VS \leftarrow  (VS)V_{S}^\prime = VSS^{-1}V^\prime T = (VV^\prime) T.
$$
This means we can can repeat our argument to show the reduction at each subsequent step is the same.
\end{proof}

Now we describe the facial reduction algorithm applied to the symmetry reduced program \cref{sdp_sys_reduced}. The facial reduction algorithm  at the $k$-th step considers  the feasible region in variables $(x,\tilde{R}_{1},\ldots,\tilde{R}_{t})$ determined by
\begin{equation}\label{sd_sym}
\begin{array}{rcl}
Ax &= &b\\
\blkdiag \left( \BBts(x) \right) &= &  \left(\tilde{V}_{1}\tilde{R}_{1}\tilde{V}_{1}^{T},\ldots,\tilde{V}_{t}\tilde{R}_{t}\tilde{V}_{t}^{T}\right)\\
\tilde{R}_{k} &\in& \mathcal{S}_{+}^{\tilde{r}_{k}},
\end{array}
\end{equation}
for some $\tilde{V} =  \Blkdiag (\tilde{V}_{1},\ldots,\tilde{V}_{t})$
with $\tilde{V}_{i} \in \mathbb{R}^{n_{i} \times \tilde{r}_{k}}$, see
also~\cref{sdp_sys_facial}.
Here $\blkdiag = \Blkdiag^*$.
In the first step, $\tilde{V}$ is the identity matrix and we obtain the feasible region $\FF_{x}$ of the symmetry reduced program \cref{sdp_sys_reduced}.

The auxiliary system for \cref{sd_sym} is to find $(y,\widetilde{W}_{1},\ldots,\widetilde{W}_{t})$ such that
\begin{equation}\label{sys_aux}
\begin{array}{rll}
A^{T}y &=& \BBt(\Blkdiag(\widetilde{W}_{1},\ldots,\widetilde{W}_{t}))\\[1ex]
0 &\neq& (\tilde{V}_{1}^{T}\widetilde{W}_{1}\tilde{V}_{1},\ldots, \tilde{V}_{t}^{T}\widetilde{W}_{t}\tilde{V}_{t} )\in (\mathcal{S}^{\tilde{r}_{1}}_+,\ldots,\mathcal{S}^{\tilde{r}_{t}}_+) \text{ and } b^{T}y \leq 0.
\end{array}
\end{equation}
Then $\Blkdiag (\tilde{V}_{1}^{T}\widetilde{W}_{1}\tilde{V}_{1},\ldots, \tilde{V}_{t}^{T}\widetilde{W}_{t}\tilde{V}_{t} )$ is an  exposing vector of the symmetry reduced problem.
 Let $\tilde{V}_{i}^\prime$ be the matrix whose independent columns span
$\Null(\tilde{V}_{i}^{T}\widetilde{W}_{i}\tilde{V}_{i})$.  In the facial reduction algorithm, we replace the
matrix $\tilde{V}_{i}$ by $\tilde{V}_{i}\tilde{V}_{i}^\prime$. Then we repeat the algorithm until the auxiliary system \cref{sys_aux} has no solution.

Our main result in this section is that the singularity degree of the
symmetry reduced \SDP~\cref{sdp_sys_reduced}
is equal to the singularity degree of the original \SDP \cref{sdp_standard}.}

\begin{theorem}
\label{thm:sdleq}
$\sd(\PP_{F_x}) = \sd(\PP_F).$
\end{theorem}
\begin{proof}
We show first the inequality $\sd(\PP_{F_x}) \leq \sd(\PP_F).$
In particular, we show that if we apply the facial reduction algorithm to the original
SDP \cref{sdp_standard}, then the solution of the auxiliary system~\cref{sdp_aux} can be used to construct a solution to the auxiliary system~\cref{sys_aux} of the symmetry reduced problem~\cref{sdp_sys_reduced}.

Let $y$ be a solution to the auxiliary system~\cref{sdp_aux} in the $k$-th facial reduction step. Let $W = \A^{*}_V(y) \in A_{\GG}$
(see \Cref{lemma:exposing_vectorinAG}) and $\widetilde{W} = Q^{T}WQ$, where $Q$ is as specified in  \Cref{wedd}.
Further, let $\widetilde{W}_{j} \in \mathcal{S}^{n_{j}}_+$ be the $j$-th block of $W$ ($j=1,\ldots,t$).

If  $k=1$ in  the \FR algorithm,  then the matrices $V$ and $\tilde{V}$ are identity matrices.
 As $W \succeq 0$, we have $\widetilde{W}_{j} \succeq 0$ ($j=1,\ldots,t$). It also holds that $b^{T}y \leq 0$ and
\begin{equation}\label{sd_eq}
\BBt( \Blkdiag(\widetilde{W}_{1},\ldots,\widetilde{W}_{t})) = \BBt(Q^{T}\A^{*}(y)Q)  = \BB(\A^{*}(y)) = A^{T}y,
\end{equation}
see \cref{eq:Qblkdiag,eq:matrixrepresA}.
Thus $(y,\widetilde{W}_{1},\ldots,\widetilde{W}_{t})$ satisfies the auxiliary system~\cref{sys_aux}.
Also, we have that $\rank \left(\A^{*}(y)\right)  = \sum_{j=1}^{t} \rank \widetilde{W}_{j}$.
Let $V$ and $\tilde{V} = \Blkdiag (\tilde{V}_{1},\ldots,\tilde{V}_{t})$ be matrices whose independent columns span $\Null(W)$ and $\Null(\widetilde{W})$, respectively. It follows from~\Cref{ro_range} that $\range(V) = \range(Q\tilde{V})$.
From~\Cref{sd_inv} it follows that we can take  $V= Q\tilde{V}$ in the next step.

Let $k > 1$
and $V= Q\tilde{V}$ where $V$ and $\tilde{V}$ are derived in the previous iterate of the \FR algorithm.
Then, we have that
$$\A_{V}^{*}(y) = V^{T}\A^{*}(y)V = \tilde{V}^{T}\left(Q^{T}\A^{*}(y)Q\right)\tilde{V} = \tilde{V}^{T}\widetilde{W}\tilde{V}$$
is block diagonal. As $\A_{V}^{*}(y) \succeq 0$, we have that each block $\tilde{V}_{j}^{T}\widetilde{W}_{j}\tilde{V}_{j}$ $(j=1,\ldots, t)$ is positive semidefinite.
It also holds that $b^Ty\leq 0$ and $\BBt(\widetilde{W}) = A^{T}y$. Thus $(y,\widetilde{W}_{1},\ldots,\widetilde{W}_{t})$ satisfies the auxiliary system~\cref{sys_aux}.
Further, we have that $\rank \left( \A_{V}^{*}(y)\right)  = \sum_{j=1}^{t} \rank (\tilde{V}_{j}^{T}\widetilde{W}_{j}\tilde{V}_{j})$.

Let $V^\prime$ and $\tilde{V}_{j}^\prime$ $(j=1,\ldots, t)$ be matrices whose independent columns span $\Null\left(\A_{V}^{*}(y)\right)$ and $\Null(\tilde{V}_{k}^{T}\widetilde{W}_{k}\tilde{V}_{k})$ $(j=1,\ldots, t)$, respectively.
As $\A_{V}^{*}(y) = \tilde{V}^{T}\widetilde{W}\tilde{V}$ is block
diagonal we can simply take $V^\prime = \tilde{V}^\prime$.
Thus after updating  $V \leftarrow  VV^\prime$ and $\tilde{V} \leftarrow  \tilde{V}\tilde{V}^\prime$, we have $V =Q \tilde{V}$ in the next step.
We can repeat the same argument until the facial reduction algorithm terminates.

Next, we show that  $\sd(\PP_{F_x}) \geq \sd(\PP_F)$. Let us assume that $(y,\widetilde{W}_{1},\ldots,\widetilde{W}_{t})$
satisfies the auxiliary system \Cref{sys_aux} in the first facial reduction step.
Recall that in the first step, $\tilde{V}$ is the identity matrix. For $Q$ defined as in \Cref{wedd}, we have that
\begin{equation}\label{sd_eq3}
W := Q\Blkdiag(\widetilde{W}_{1},\ldots,\widetilde{W}_{t})Q^{T} \in A_{\GG}.
\end{equation}
To show that $y$ satisfies the auxiliary system \cref{sdp_aux}, we  have to prove that  $\A^{*}(y) \succeq 0$.
It holds that
\begin{equation}\label{sd_eq2}
\BB(\A^{*}(y)) = A^{T}y = \BBt(\Blkdiag(\widetilde{W}_{1},\ldots,\widetilde{W}_{t})) = \BB(W),
\end{equation}
see also \cref{sd_eq}.
The second equality above uses the feasibility of \Cref{sys_aux}.
Since we have that $\A^{*}(y) \in A_{\GG}$
and  $W \in A_{\GG}$, it follows from \cref{sd_eq2} that $W = \A^{*}(y)$, and from
\cref{sd_eq3} and \Cref{sys_aux} that $\A^{*}(y) \succeq 0$.
Recall that we assumed that the data matrices of the \SDP problem \eqref{sdp_standard} are contained in the matrix $*$-algebra  $A_{\GG}$, see \Cref{sec_group}.

Let $k > 1$. Using $V = Q\tilde{V}$ where $V$ and $\tilde{V}$ are derived in the previous iterate of the \FR algorithm, \Cref{sys_aux,sd_eq3}, we have  that
	$$0 \preceq \tilde{V}^{T} \Blkdiag(\widetilde{W}_{1},\ldots,\widetilde{W}_{t}) \tilde{V} = V^{T}WV = V^{T}\A^{*}(y)V = \A_{V}^{*}(y).$$
In addition, it follows from construction of $W$ and \Cref{sd_inv} that we can take $V$ and $\tilde{V}$ such that $V=Q\tilde{V}$ in the next \FR step.

\end{proof}

The following~\Cref{cor:sdone} follows
from~\cite[Theorem 3.2]{Tun:01} in that the linear
manifold is represented by a concrete constraint and is applied to finding
an exposing vector. More precisely, if we can find the affine span of
our original feasible set in the ground space, then we can always find
the representation using a linear mapping as in~\cref{eq:liftedgrset}.
This means we can always find the appropriate exposing vector and obtain
singularity degree one, see also~\cite{DrusWolk:16}.
Note that this includes the hard combinatorial problems we work with below.
\begin{corollary}
\label{cor:sdone}
Consider the quadratic model as given in~\Cref{thm:FRalaTUNforDNN}, and
suppose that the matrix $A$ is part of the given data of the problem,
Then the singularity degree is one.
\end{corollary}
\begin{proof}
The proof uses $A$ to construct the exposing vector. Therefore, one step
of the \FR algorithm is sufficient, see~\cref{eq:liftedgrset}.
More precisely, the linear constraint in the ground set is lifted into the \SDP as
in~\cref{eq:liftedgrset}.
\end{proof}

\begin{remark}
\label{rem:sdSDPDNN}
The definition of singularity degree can be extended to \DNNp,
and to a general cone optimization problem, to be the minimum number of
steps in the \FR~\cite[Algor. B]{bw3}. Here this means we continue to
find the minimum number of steps with nonzero exposing vectors.
An interesting question is to find the relationship between the singularity degree
of the \SDP and the \DNNp. It appears that the \DNN is at most one more than for \SDPp.
Moreover it is known that the singularity degree of the $\DNN^n$  is at most $n$, see \cite[Corollary 20]{MR3845279}.
\end{remark}

\subsection{Simplifications}
\label{sect:simplif}
After \FR, some of the constraints become redundant in the facially reduced program~\cref{sdp_facial}.
We show here that the same constraints are also redundant in the facially and symmetry reduced program~\cref{sdp_sys_facial}.
Proof of \Cref{lemma:simplisy} is clear.

\begin{lemma} \label{lemma:simplisy}
For any subset $\mathcal{I} \subseteq [m] : = \{1,\ldots,m\}$, we define the spectrahedron
$$\mathcal{F}(\mathcal{I}):=\{ X \in \mathcal{S}^{n} \;|\; \langle A_i,X \rangle=b_i ~~\forall i\in \mathcal{I}, \; X = VRV^{T}, \; R \in \mathcal{S}^{r}_{+} \}.$$
If the constraints in $[m]\backslash \mathcal{I}$ are redundant,
e.g., $\mathcal{F}([m]) = \mathcal{F}(\mathcal{I})$, then
$\mathcal{F}([m]) \cap A_{\GG} = \mathcal{F}(\mathcal{I}) \cap A_{\GG}$.
\end{lemma}
Although a proof of \Cref{redundant} follows directly from \Cref{lemma:simplisy},
we provide it due to our intricate notation.
\begin{corollary}\label{redundant}
	Let $\mathcal{I} \subsetneq \{1,\ldots,m\}$.
	Suppose that the constraints $\langle A_{k},VRV^{T} \rangle = b_{k}, k \notin
	\mathcal{I}$, are redundant in~\cref{sdp_facial}, i.e.,~the facially
	reduced formulation~\cref{sdp_facial} is equivalent to
	\begin{equation}\label{sdp_facial_r}
	\min_{R \in \Ss^r_+} \{ \langle V^{T}CV,R \rangle \;|\; \langle V^{T}A_{i}V,R \rangle  = b_{i},  \;\; \forall i \in \mathcal{I}  \}.
	\end{equation}
	Then the constraints
	\[
	\sum_{j=1}^{d} A_{k,j}x_{j} = b_{k}, k
	\notin \mathcal{I},
	\]
	are redundant in~\cref{sdp_sys_facial}, i.e.,~the
	facially and symmetry reduced program~\cref{sdp_sys_facial} is equivalent to
	\begin{equation}
	\label{sdp_sys_facial_r}
	\min_{x \in \R^{d}, \tilde{R} \in \Ss^r_+ } \{ c^{T}x \;|\; \sum_{j=1}^{d} A_{i,j}x_{j} = b_{i}, \forall i \in \mathcal{I} ,
	\;\; \BBts(x) = \tilde{V}\tilde{R}\tilde{V}^{T}\}.
	\end{equation}
\end{corollary}

\begin{proof}
Let  $Q$ be specified in \Cref{wedd}.
Since  $\range(V) =  \range(Q\tilde{V})$, see Lemma \ref{ro_range},
 we  assume w.l.g.~that $V$ in \cref{sdp_facial_r}
satisfies $V = Q\tilde{V}$.
Let $(x,\tilde{R})$ be feasible for \cref{sdp_sys_facial_r}. Define $X := \sum_{i=1}^{d}B_{i}x_{i}$.
We show the equivalence in the following order:
$$\begin{array}{rcl}
\fbox{$(x,\tilde{R})$ \text{feasible
		for} \cref{sdp_sys_facial_r}}
&\implies &  \fbox{$\tilde{R}$
	\text{feasible for} \cref{sdp_facial_r}}
\\ &\implies &
\fbox{$\tilde{R}$ \text{feasible for} \cref{sdp_facial}}
\\ &\implies &
\fbox{$(x,\tilde{R})$ \text{feasible for} \cref{sdp_sys_facial}}.
\end{array}$$
Since
$
Q^{T}( \sum_{i=1}^{d}B_{i}x_{i})Q = \tilde{V}\tilde{R}\tilde{V}^{T} = Q^{T}V\tilde{R}V^{T}Q,
$
we have $\sum_{i=1}^{d}B_{i}x_{i} = V\tilde{R}V^{T}$. Using the feasibility and \eqref{eq:matrixrepresA}, it holds that
\[
\langle A_{i}, V\tilde{R}V^{T} \rangle = \langle A_{i}, \sum_{j=1}^{d}B_{j}x_{j}\rangle = \sum_{j=1}^{d} A_{i,j}x_{j} = b_{i}, \forall i \in \mathcal{I}.
\]
Thus all the linear constraints in \cref{sdp_facial_r} are satisfied, and $\tilde{R} \succeq 0$ is feasible for \cref{sdp_facial_r}.
By assumption, $\tilde{R}$ is also feasible for \cref{sdp_facial}. Thus the
constraints  $\langle A_{i}, V\tilde{R}V^{T} \rangle = b_{i}, \forall i \notin
\mathcal{I}$ are satisfied as well.
This shows that $(x,\tilde{R})$ is feasible for~\eqref{sdp_sys_facial}.
\end{proof}

To obtain a formulation for the facially and symmetry reduced program~\cref{sdp_sys_facial} in variable $\tilde{R}$ only, we can replace $x$ in terms of $\tilde{R}$ using the constraint $\BBts(x) = \tilde{V}\tilde{R}\tilde{V}^{T}$.
This substitution can be done easily by rewriting the constraints as
$$b_{i} = \langle A_{i}, X \rangle = \langle Q^{T}A_{i}Q, Q^{T}XQ \rangle  = \langle Q^{T}A_{i}Q, \BBts(x) \rangle = \langle Q^{T}A_{i}Q, \tilde{V}\tilde{R}\tilde{V}^{T} \rangle.$$
The objective can be similarly changed. This method, however,
does not work for \DNN relaxations.
This difficulty can be resolved as follows.

\begin{theorem}\label{uni_redu}%
Consider the facially and symmetry reduced  relaxation~\cref{sdp_sys_facial} with nonnegativity constraints,
\begin{align}
\label{dnn_sys_facial}
\begin{split}
& \min \{ c^{T}x \;|\; \textdef{$Ax=b$},
\;\;   \BBts(x) = \tilde{V}\tilde{R}\tilde{V}^{T}, \; \tilde{R} \succeq
	0, \; x \geq 0\},
\end{split}
\end{align}
where $\tilde{R} =  \Blkdiag(\tilde{R}_{1},\ldots,\tilde{R}_{t})$.
Equate $x$ with
\[
x\leftarrow f(\tilde{R}) ={\Diag(w)^{-1}}\BB(V\tilde{R}V^{T}),
\]
 where
\[
w = (\langle B_{i}, B_{i} \rangle )_{i} \in
\mathbb{R}^{d},\, V = Q\tilde{V},
\]
and $Q$ is specified in \Cref{wedd}.
Then~\cref{dnn_sys_facial} is equivalent to
\begin{align}
\label{dnn_sys_facial2}
\begin{split}
& \min \{ c^T\!f(\tilde{R}) \;|\; \textdef{$Af(\tilde{R})=b$},
\;\;  \tilde{R} \succeq 0, \; f(\tilde{R}) \geq 0\},
\end{split}
\end{align}
where diagonal blocks of the block diagonal matrix $\tilde{R}$
are set to be equal for the corresponding repeating blocks in  $Q^{T}A_{\GG}Q$.
\end{theorem}
\begin{proof}
If $(x,\tilde{R})$ is feasible for~\cref{dnn_sys_facial}, then $\BBs(x)
= V\tilde{R}V^{T}$. As $w>0$ and $\BB \BBs = \Diag(w)$, we have $x =
 f(\tilde{R})$ and thus $\tilde{R}$ is feasible for~\cref{dnn_sys_facial2}.

Let $\tilde{R}$ be feasible for~\cref{dnn_sys_facial2}.
Since $\tilde{V}\tilde{R}\tilde{V}^{T}$ is a block-diagonal matrix in the
 algebra $Q^{T}A_{\GG}Q$, we have $V\tilde{R}V^{T}  =
Q\tilde{V}\tilde{R}\tilde{V}^{T}Q^{T} \in A_{\GG}$. It follows
from~\Cref{wedd} that there exists a unique $x$ such that
$V\tilde{R}V^{T} = \BBs(x)$. Then we must have $x = f(\tilde{R})$ and
thus $(x,\tilde{R})$ is feasible for~\cref{dnn_sys_facial}.
\end{proof}

For the Hamming scheme, we have an explicit expression for the
orthogonal matrix $Q$ used in~\Cref{uni_redu}, see~\Cref{sec:Hamming}
and~\Cref{sec_qap}. In general, we do not know the corresponding
orthogonal matrix explicitly. In~\Cref{sec_mc}, we use the  heuristics
from~\cite{MR2546331} to compute a  block diagonalization of $A_\GG$.
In this case, the equivalence in~\Cref{uni_redu} may not be true, and
\eqref{dnn_sys_facial2} may be weaker than~\cref{dnn_sys_facial}.
However our computational results indicate that all the bounds remain
the same, see~\Cref{table:QueenRez} below.

	\section{The alternating direction method of multipliers, \ADMMp}
\label{sec_admm}
	
It is well known that interior-point methods do not scale well for \SDPp.
Moreover, they have great difficulty with handling additional cutting
planes such as nonnegativity constraints. In particular, solving the
doubly nonnegative relaxation, \DNNp, using interior-point methods is extremely difficult.
The alternating direction method  of multipliers  is a first-order method for convex problems developed in the 1970s, and rediscovered recently.
This method decomposes an optimization problem into subproblems that may
be easier to solve. In particular, it is extremely successful for
splittings with two cones.
This feature makes the \ADMM well suited for our large-scaled \DNN problems. For
state of the art in theory and applications of the \ADMMp, we refer the
interested readers to~\cite{BoydParikhChuPeleatoEckstein:11}.

Oliveira, Wolkowicz and Xu~\cite{OliveiraWolkXu:15} propose a version of
the \ADMM for solving an \SDP relaxation for the Quadratic Assignment Problem (\QAP).
Their computational experiments show that the proposed variant of the \ADMM exhibits remarkable robustness, efficiency, and even provides improved bounds.

\subsection{Augmented Lagrangian}
\label{sect:auglagr}
We modify the approach from~\cite{OliveiraWolkXu:15} for solving our
\SR and \FR reduced \DNN relaxation~\cref{sdp_sys_facial}.
We have a greatly simplified structure as we applied \SR to the \SDP
relaxation, and we were then able to move
the nonnegativity constraints to a simple
vector $x\geq 0$ contraint. We in particular obtain a
more efficient approach for solving the $x$-subproblem.

\medskip

Let $\tilde{V} = \Blkdiag(\tilde{V}_{1},\ldots,\tilde{V}_{t})$ and
$\tilde{R} = \Blkdiag(\tilde{R}_{1},\ldots,\tilde{R}_{t})$. The augmented Lagrangian
of~\cref{sdp_sys_facial} corresponding to the linear
	constraints $\BBts(x) = \tilde{V}\tilde{R}\tilde{V}^{T}$ is given by:
	\begin{align*}
	{\LL}(x,\tilde{R},\tilde{Z})  = \langle \tilde{C}, \BBts(x)\rangle +
\langle \tilde{Z}, \BBts(x) - \tilde{V} \tilde{R}\tilde{V}^{T} \rangle + \frac{\beta}{2}
||\BBts(x) - \tilde{V}\tilde{R}\tilde{V}^{T}||^{2},
	\end{align*}
where, see~\cref{sdp_afterBts}, $\tilde{C} = Q^{T}C Q$
is a block-diagonal matrix as $C \in A_\GG$,  $\tilde{Z}$ is also in
block-diagonal form, and $\beta >0$ is the penalty parameter.
	
The alternating direction method of multipliers, \ADMMp,
uses the \textdef{augmented Lagrangian, ${\LL}(x,\tilde{R},\tilde{Z})$}, and
essentially solves the max-min problem
\[
\max_{\tilde{Z}} \min_{x\in P,\tilde{R}  \succeq 0}{\LL}(x,\tilde{R},\tilde{Z}),
\]
where $P$ is a simple polyhedral set of constraints on $x$, e.g.,~linear
constraints $Ax=b$ and nonnegativity constraints, see~\cref{psim_def} below.
\index{$P$, polyhedral constraints on $x$}
\index{polyhedral constraints on $x$, $P$}
The advantage of the method is the simplifications obtained for
the constraints by taking advantage of the splitting in the variables. We then
find the following updates $(x_+,\tilde{R}_+,\tilde{Z}_+)$:
\index{${\LL}(x,R,Z)$, augmented Lagrangian}
\begin{align}\nonumber
	x_+ & = \arg \min_{x  \in P} {\LL}(x,\tilde{R},\tilde{Z}), \nonumber \\
	\tilde{R}_+ & = \arg \min_{\tilde{R}  \succeq 0 } {\LL}(x_+,\tilde{R},\tilde{Z}),\nonumber  \\
	\tilde{Z}_+ & = \tilde{Z} + \gamma   \beta   (\BBts(x_+) - \tilde{V}\tilde{R}_+\tilde{V}^{T}). \nonumber 
	\end{align}
Here, $\gamma \in(0, \frac{1+\sqrt{5}}{2})$ is the step size for updating the dual variable $\tilde{Z}$.
	In the following sections we explain in details how to solve each subproblem.
	
\subsection{On solving the $\tilde R$-subproblem}

	The $\tilde R$-subproblem can be explicitly solved. We complete the square
	and  get the equivalent problem
	\begin{equation}\label{admm_Rsub}
	\begin{array}{cll}
	\tilde R_+ &=	& \min\limits_{\tilde R \succeq 0}  ||\BBts(x) -
\tilde{V}\tilde R\tilde{V}^{T} + \frac 1{\beta} \tilde{Z}||^{2} \\[1.5ex]
	&= &  \min\limits_{\tilde R \succeq 0}  ||\tilde R - \tilde{V}^{T}(\BBts(x)+ \frac 1{\beta} \tilde{Z})\tilde{V}||^{2} \\[1.5ex]
	&= &  \sum_{k=1}^{t} \min\limits_{\tilde R_{k} \succeq 0}
||\tilde R_{k} -
\big(\tilde{V}^{T}(\BBts(x)+ \frac 1{\beta} \tilde{Z})\tilde{V}\big)_{k}||^{2}.
	\end{array}
	\end{equation}
	Here, we normalize each block $\tilde{V}_k$ such that $\tilde{V}_k^T \tilde{V}_k=I$, and thus
 $\big(\tilde{V}^{T}(\BBts(x)+ \frac 1{\beta} \tilde{Z})\tilde{V}\big)_{k}$ is the $k$-th block of $\tilde{V}^{T}(\BBts(x)+ \frac 1{\beta} \tilde{Z})\tilde{V}$
	corresponding to $\tilde R_{k}$.
So we only need to solve $k$ small problems whose optimal solutions are
$$
\tilde R_{k} =
\mathcal{P}_{\mathbb{S}_+}\left(\tilde{V}^{T}(\BBts(x)+ \frac
1{\beta} \tilde{Z})\tilde{V}\right)_{k}, \quad k=1,\ldots,t,$$
	where $\mathcal{P}_{\mathbb{S}_+}(M)$ is the projection
onto the cone of positive semidefinite matrices.
	
	\index{$\mathcal{P}_{\mathbb{S}_+}(\cdot)$, projection onto
positive semidefinite matrices}
	\index{projection onto positive semidefinite matrices,
$\mathcal{P}_{\mathbb{S}_+}(\cdot)$}

	\subsection{On solving the $x$-subproblem} \label{subSect:xSub}
	For the $x$-subproblem, we have
	\begin{equation}\label{xsub}
	x_+ =  \arg \min\limits_{x  \in P}  \left\|\BBts(x) -
\tilde{V}\tilde{R}\tilde{V}^{T} + \frac{\tilde{C}+ \tilde{Z}}{\beta}\right\|^{2}.
	\end{equation}
	For many combinatorial optimization problems, some of the
constraints $Ax=b$ in $\eqref{sdp_sys_reduced}$ become redundant after
\FR of their semidefinite programming relaxations, see~\Cref{redundant}.
Thus, the set $P$ often collapses to a simple set. This often leads to an
analytic solution for the $x$-subproblem; e.g.,~this happens for the
quadratic assignment, graph partitioning, vertex separator, and
shortest path problems.

	For some interesting applications, the $x$-subproblem is
equivalent to the following special case of the weighted, relaxed,
quadratic knapsack problem:
	\begin{equation}\label{psim_def}
	\begin{array}{cll}
	\min_{x}  & \frac{1}{2}||{\TT} ^*(x)- Y ||^{2} \\[1.5ex]
	\text{s.t.} & x\in P:= \{x \,|\,  w^{T}x = c,\, x\geq 0\},
	\end{array}
	\end{equation}
	where $Y$ is a given matrix and ${\TT} ^*(x) = \sum_{i=1}^{q} x_{i}T_{i}$ for some given symmetric matrices $T_{i}$.
The problem~\cref{psim_def} is a projection onto the weighted simplex.
We consider the following assumption on a linear transformation
${\TT}: \Sn \to \Rq$ and its adjoint.
\begin{assump}
\label{psim_ass}
The linear transformation ${\TT}: \Sn \to \Rq$ in~\Cref{psim_def} satisfies
$$
{\TT} ({\TT} ^*(x)) = \Diag(w)x, \, \forall x\in \Rq,\text{  for some  } w>0.
$$
\end{assump}

\begin{lemma}
Suppose that the linear transformation $\TT$
satisfies~\Cref{psim_ass}, and that~\Cref{psim_def} is feasible.
Then the projection problem~\cref{psim_def} can be solved efficiently
(explicitly) using~\Cref{prj_sim}.
\end{lemma}
\begin{proof}
	The Lagrangian function of the problem is $$\frac{1}{2}||
{\TT} ^*(x) - {Y}||^{2} - \tau(w^{T}x - c) - \lambda^{T}x,$$
	where $\tau \in \R$ and $\lambda \in \R_+^{q}$ are the
Lagrangian multipliers.  The KKT optimality conditions for the problem
are given by
	$$\begin{array}{rl}
	{\TT} ({\TT} ^*(x)) - {\TT} ({Y}) - \tau w - \lambda &= 0, \\
	x & \geq 0, \\
	\lambda & \geq 0, \\
	\lambda^{T}x & = 0, \\
	w^{T}x &= c.
	\end{array}$$
	
Note that $\Diag(w)$ is the matrix representation of $\TT\circ\TT^*$.
This means that $ \langle T_i,T_j \rangle =0, \forall i\neq j$,
and we can simplify the first condition.\footnote{Note that this is always satisfied for basis matrices from a coherent configuration.}
This yields
\[
x_{i} =  w_{i}^{-1}({\TT} ({Y}))_{i} + \tau +
w_{i}^{-1}\lambda_{i}.
\]
Define the data vector $y := {\TT} ({Y})$.
The complementary slackness $\lambda^{T}x = 0$ implies that if $x_{i} > 0$, then $\lambda_{i} = 0$  and $x_{i} = w_{i}^{-1}y_{i} + \tau$.
If $x_{i} = 0$, then $w_{i}^{-1}y_{i} + \tau = -w_{i}^{-1}\lambda_{i} \leq 0$.
Thus the zero and positive entries of the optimal solution $x$ correspond to the smaller than $-\tau$ and the larger than $-\tau$ entries of
 $(w_{i}^{-1}y_{i})_{i=1}^{q}$, respectively.
	
Let us assume, without loss of generality, that $(w_{i}^{-1}y_{i})_{i=1}^{q},x$ are sorted in non-increasing order:
\[
\frac{y_{1}}{w_{1}} \geq  \ldots \geq  \frac{y_{k}}{w_{k}} \geq
\frac{y_{k+1}}{w_{k+1}} \geq  \ldots \geq \frac{y_{q}}{w_{q}}, \quad
x_{1} \geq  \ldots \geq  x_{k} > x_{k+1} =  \ldots =x_{q} = 0.
\]
 The condition $w^{T}x = c$ implies that
	$$ w^{T}x = \sum_{i=1}^{k}w_{i}(\frac{y_{i}}{w_{i}} + \tau)  = \sum_{i=1}^{k}y_{i} + \tau \sum_{i=1}^{k}w_{i} = c,$$
	and thus $$\tau = \frac{c- \sum_{i=1}^{k}y_{i}}{\sum_{i=1}^{k}w_{i}}.$$
	Therefore, one can solve the problem by simple inspection once $k$ is known.
	The following algorithm finds an  optimal solution $x$ to the
problem~\cref{psim_def}. The correctness of the algorithm is then similar to the projection onto  the (unweighted) simplex problem, see~\cite{chen2011projection,condat2016fast}.
\end{proof}
		
	\begin{samepage}
\hrulefill
\vspace{-.06in}
	\begin{alg}[Finding an optimal solution for~\cref{psim_def}]
\label{prj_sim}
		\begin{algorithmic}
\mbox{ }
\State \textbf{Input:} $w\in \R^{q},y \in \R^{q}$
\State \quad Sort $\{y_i/w_i\}$ such that $y_{1}/w_{1} \geq  \ldots \geq y_{q}/w_{q}$
\State \quad Set $k: = \max_{1 \leq k \leq n} \{ k \;|\;  w_{k}^{-1}y_{k} + (\sum_{i=1}^{k}w_{i})^{-1}(c- \sum_{i=1}^{k}y_{i}) > 0 \}$
			\State \quad Set $\tau := (\sum_{i=1}^{k}w_{i})^{-1}(c- \sum_{i=1}^{k}y_{i})$
			\State \quad Set $x_{i} = \max \{w_{i}^{-1}y_{i} + \tau ,0\}$ for $i=1,\ldots q$
			\State \textbf{Output:} $x  \in \R^{q}$
		\end{algorithmic}
	\end{alg}
\vspace{-.15in}
\hrulefill
\vspace{.1in}
\end{samepage}

	In our examples, see~\Cref{sec_qap,sec_mc}, the $x$-subproblem~\cref{xsub} satisfies
	\Cref{psim_ass}. Moreover, we have the following lemma.
We remind the reader that $J$ denotes the matrix of all ones.
	\begin{lem}
\label{eq:xsubwJ}
	The $x$-subproblem~\cref{xsub} satisfies~\Cref{psim_ass}, if
$$P = \{x \in \R^{q} \; | \; \langle J, \BB^*(x) \rangle = c , x \geq 0\}.$$
	\end{lem}
	\begin{proof}
	It holds that
	\begin{equation}\label{BBs}
\big(\tilde{B}(\tilde{B}^*(x))\big)_{i} =  \langle \tilde{B}_{i} , \sum_{j=1}^q \tilde{B}_{j}x_{j} \rangle = \langle \tilde{B}_{i},
\tilde{B}_{i} x_{i} \rangle = \trace (Q^{T}B_{i}^{T}QQ^{T}B_{i}Q) x_{i} = w_{i} x_{i},
\end{equation}
where $w_i = \trace (B_{i}^{T}B_{i})$. Furthermore, $\langle J, \BB^*(x)
\rangle  = w^{T}x$ with $w = (w_{i}) \in \R^{q}$.
 Thus we set $\TT = \BB$ and note that
${\TT } ({\TT } ^*(x)) = \Diag(w)x$.
	\end{proof}

\section{Numerical results}
\label{sec_numers}
We now demonstrate the efficiency of our new
approach on two classes of problems:
the quadratic assignment problem, \QAPp, and several types of
graph partitioning problem, \GPp.

Our tests were on: Dell PowerEdge M630 computer; two Intel
Xeon E5-2637v3 4-core 3.5 GHz (Haswell) CPU; 64GB memory; linux.
The interior point solver was Mosek, see~\cite{andersen2000mosek}.
We had to use a different computer to accommodate some of the larger
problems when using an interior point approach,
see description of~\Cref{table:GPPRezthree}.

We report only the ADMM solver time, since the preprocessing time is very small in comparison with the ADMM part.
In particular, we find exposing vectors  explicitly in all our examples.
Further, for most of the test instances we know generators of automorphism groups as well as the orthogonal matrix $Q$.
For instances for which we need to find generators  e.g.,  the instances in~\Cref{table:graphInfo}  the preprocessing time is less than one second.

The stopping conditions and tolerances
are outlined at the start of~\Cref{sect:numericsQAP},
in~\Cref{def:tolerances}.
Our results include \emph{huge} problems of sizes up to $n=512$ for the \QAPp,
yielding of the order $n^2$ \SDP constraints and
$n^4$ nonnegativity constraints.\footnote{The link to the codes for the \QAP can be found on the webpage
\url{www.huhao.org}. The codes for the other problems require finding
symmetries in the graph; and therefore these codes
and details can be obtained upon request directly from the authors.} \label{pg:codes}

\subsection{The quadratic assignment problem, \QAPp}
\label{sec_qap}
\index{quadratic assignment problem, \QAPp}
\index{\QAPp, quadratic assignment problem}
	
\subsubsection{Background for the \QAPp}
\label{sect:backgrQAP}
The Quadratic Assignment Problem  was introduced in $1957$ by Koopmans and Beckmann as a model for location problems that take into
account  the linear cost of placing a new facility on a certain site,
plus the quadratic cost arising from the product
of the flow between facilities and distances between sites.
The \QAP contains the traveling salesman problem as a special case and
is therefore NP-hard in the strong sense. It is generally considered to
be one of the \emph{hardest} of the NP-hard problems.

\index{$\Pi$, permutation matrices order $n$}
\index{permutation matrices order $n$, $\Pi$}

Let $A,B\in \Sn$, and let $\Pi_n$ be the set of $n \times n$
permutation matrices. The \QAP (with the linear term with
appropriate $C$ in brackets)
can be stated as follows:
\[
\min\limits_{X \in \Pi_n} \tr(AX^TBX) \quad (+\trace(X^TC)).
\]
The \QAP is extremely difficult to solve to optimality,
e.g.,~problems with $n\geq 30$ are still considered hard.
It is well known that \SDP relaxations provide
strong bounds, see e.g.,~\cite{KaReWoZh:94,MR2546331}.
However even for sizes $n\geq 15$, it is difficult to solve the
resulting \SDP relaxation by interior point methods if one cannot
exploit special structure such as symmetry.
Solving the \DNN relaxation is significantly more difficult.

Here, we first consider the \DNN relaxation for the \QAP
from Povh and  Rendl~\cite{MR2532462}, i.e.,
	\index{$u_i$}
\begin{equation}\label{qap_sdp}
	\begin{array}{rl}
	\min  & \tr(A\otimes B) Y  \\
	\text{s.t.}  &  \langle J_{n^{2}},Y \rangle = n^{2} \\
	& \langle I_n \otimes (J_n-I_n) + (J_n-I_n) \otimes I_n,Y \rangle = 0 \\
	&  \langle I_n \otimes E_{ii}, Y \rangle = 1,  \forall i = 1,\ldots,n\\
	&  \langle E_{ii} \otimes I_n, Y \rangle = 1, \forall i = 1,\ldots,n\\
	& Y \succeq 0, Y \geq 0,  \quad (Y\in \DNNp)
	\end{array}
	\end{equation}
where
and \textdef{$E_{ii}=u_iu_i^T$}, where $u_i\in \R^n$ is $i$-th unit vector.
The authors  in~\cite[Theorem 7.1]{MR2546331} show that one can take
\index{$\aut(A)$, automorphism group of $A$}
\index{automorphism group of $A$, $\aut(A)$}
\begin{equation} \label{AGforQAP}
\A_{\GG} = {\A} _{{\aut}(A)} \otimes  {\A} _{{\aut}(B)},
\end{equation}
where  ${\aut}(A):= \{ P \in \Pi_n: AP=PA \}$ is the automorphism group of $A$.

\begin{remark}
\label{rem:DNNPovhRendl}
The \DNN relaxation~\cref{qap_sdp} is known to be theoretically
equivalent, yielding the same optimal value,
to the \DNN relaxation denoted (QAP$_{R3}$) in
Zhao et al.~\cite{KaReWoZh:94}.
The constraints $\langle I_n \otimes (J_n-I_n) + (J_n-I_n) \otimes I_n,Y \rangle = 0$
 are generally called the \textdef{gangster constraints}, see \Cref{qap_facial}.
The third and fourth lines of constraints in~\cref{qap_sdp} arise from the row and column sum constraints.

Recall that \textdef{$\svec$} is the linear transformation that vectorizes symmetric matrices,~\cite{KaReWoZh:94}.
We define \textdef{$\gsvec$ to do this vectorization of symmetric
matrices while ignoring the elements set to zero by the gangster
constraints.}
Then we can eliminate the gangster constraints
completely and replace the \DNN constraints  to get the equivalent
problem to~\cref{qap_sdp}:
\begin{equation}
\label{qap_sdpreduced}
	\begin{array}{rll}
	\min  & \gsvec(A\otimes B)^T y  \\
	\text{s.t.}  &
           \gsvec(J_{n^{2}})^Ty  = n^{2} \\
&  \gsvec( I_n \otimes E_{ii})^T y  = 1, & \forall i = 1,\ldots,n\\
&  \gsvec( E_{ii} \otimes I_n)^T y  = 1,& \forall i = 1,\ldots,n\\
	& \gsvec^*(y) \succeq 0,\, y \geq 0.
	\end{array}
	\end{equation}
This form is now similar to our final \SR reduced form before \FRp,
see~\cref{sdp_sys_reduced};
and this emphasizes that the \DNN can be represented in a split form.
\end{remark}

In the following lemma we derive the null space of the feasible
solutions of~\cref{qap_sdp}, see also Corollary 2.2.7 in~\cite{truetsch2014semidefinite}.

\begin{lemma}\label{qap_exp_lemma1}
Let $U := \frac{1}{\sqrt{n}}(nI-J) \in \R^{n \times n}$,
and let $Y$ be in the relative interior of the feasible set
of~\cref{qap_sdp}. Then
\[
\nul(Y) = \range\left(
\begin{bmatrix} U
\otimes e_{n} & e_{n} \otimes U \end{bmatrix}\right).
\]
	\end{lemma}
	\begin{proof}
		Let $X \in \Pi_n$. Then $Xe_n = X^{T}e_n = e_n$, and thus
		$$\begin{array}{lr}
		(U \otimes e_{n})^T\text{vec}(X) = U^{T}e_n = 0,\\
		 (e_{n} \otimes U)^{T} \text{vec}(X) = U^{T}e_n = 0.
		\end{array}$$
Thus $\range\left(\begin{bmatrix} U \otimes e_{n} &
e_{n} \otimes U \end{bmatrix}\right) \subseteq \nul(\hat Y)$,
where
\[
\hat{Y} = \frac{1}{n!} \sum\limits_{X\in \Pi_n} {\rm vec}(X){\rm vec}(X)^T = \frac{1}{n^2}( J\otimes J ) + \frac{1}{n^2(n-1)}(nI-J)\otimes  ( nI-J ).
\]
It is proven in~\cite{truetsch2014semidefinite} that $\hat{Y}$ is
 in the relative interior of the feasible set of~\cref{qap_sdp}.
 Recall that every matrix $Y$ in the relative interior of a face has the same null space. This shows that $\nul(Y)  \supseteq \range\left(
	\begin{bmatrix} U
		\otimes e_{n} & e_{n} \otimes U \end{bmatrix}\right)$.

It remains to show that
\[
\dim \left( \range \left(
\begin{bmatrix} U\otimes e_{n} & e_{n} \otimes U \end{bmatrix}
\right)\right)
= 2(n-1).
\]
To see this, we choose the square submatrix of size $2n-1$ associated to
\[
\text{rows: } \{kn \;|\; k = 1,\ldots,n-1\} \cup \{ n(n-1)+1,\ldots,
n^2-1\};
\quad
\text{cols: } \{1,\ldots,n-1\} \cup \{n+1,\ldots,2n\}.
\]
It has the form
$$\frac{1}{\sqrt{n}}(nI-J) \in \mathcal{S}^{2(n-1)}.$$
This square submatrix clearly has rank $2(n-1)$, and thus the statement
follows.
	\end{proof}
Let us now derive an exposing vector of the \SDP
relaxation~\cref{qap_sdp} ignoring the nonnegativity, as we have
shown we can add the nonnegativity on after the reductions.
	\begin{lemma}\label{qap_exp_lem}
Consider~\cref{qap_sdp} without the nonnegativity constraints. Then
		\begin{equation}\label{qap_exp_inv2}
		\begin{array}{rl}
			W =  I_n \otimes nJ_n + J_n \otimes (nI_n-2J_n) \in
\A_{\GG} \subseteq \mathcal{S}^{n^{2}}_+,
		\end{array}
		\end{equation}
and is an exposing vector of rank $2 (n-1)$ in $\A_{\GG}$.
	\end{lemma}
	\begin{proof}
		Let $U$ be defined as in~\Cref{qap_exp_lemma1}. Using the properties of the Kronecker product, we have
		\begin{equation}\label{qap_exp_inv}
		\begin{array}{rl}
		0\preceq W= & \begin{bmatrix} U\otimes e_{n} & e_{n} \otimes U \end{bmatrix} \begin{bmatrix} U\otimes e_{n} & e_{n} \otimes U \end{bmatrix} ^{T} \\
		 = & (UU^{T}) \otimes J + J \otimes (UU^{T}) \\
		 = & (nI-J) \otimes J + J \otimes (nI-J) \\
		 = & I \otimes nJ + J \otimes (nI-2J),
		\end{array}
		\end{equation}
as $UU^{T} = nI-J$.  From~\Cref{qap_exp_lemma1}, we have $W$ is an exposing vector of rank $2(n-1)$.
Let $P$ be any permutation matrix of order $n$. Then $P^{T}(UU^{T})P = UU^{T}$ by construction.
We now have $(P_1\otimes P_2)^{T}W(P_1\otimes P_2) = W$, for any $P_1,P_2\in  \Pi_n$; and thus $W \in \A_{\GG}$.
	\end{proof}

In the rest of this section
we show how to do \FR for the symmetry reduced program
of~\cref{qap_sdp}. We continue to add on nonnegativity constraints to
\SDP relaxations as discussed above.
The facially reduced formulation of \cref{qap_sdp} is also presented  in~\cite{truetsch2014semidefinite}.
We state it here for later use.

\begin{lemma}[{\cite{truetsch2014semidefinite}}]\label{qap_facial}
The facially reduced program of the
\textdef{doubly nonnegative, \DNNp} \cref{qap_sdp} is given by
\index{\DNNp, doubly nonnegative}
	\begin{equation}\label{qap_sdp_facial}
	\begin{array}{cl}
\min & \left\langle  \left(V^{T} \left(A\otimes B \right) V \right),R\right\rangle \\
	\text{s.t.} 	& \left\langle V^{T}JV ,R \right\rangle = n^2 \\
	& \GG(VRV^{T}) = 0 \\
	& VRV^{T} \geq 0 \\
	& R  \in \mathcal{S}^{(n-1)^2+1}_+,
	\end{array}
	\end{equation}
where, by abuse of notation,
$\GG: \mathcal{S}^{n^2} \to \mathcal{S}^{n^2}$ is a linear operator  defined by $\GG(Y) := (J- (I\otimes(J-I) + (J-I)\otimes I) ) \circ Y$
\footnote{We use $\GG$ as the group and as a linear operator, usually referred to as the \emph{gangster operator}, since the meaning is clear from the context.
Here $ \circ $ denotes the Hadamard product.},
and
the columns of $V \in \R^{n^2 \times(n-1)^2+1}$ form a basis of
the null space of $W$, see~\Cref{qap_exp_lem}.
\end{lemma}
Note that the constraints $\langle I \otimes E_{ii}, Y \rangle = 1$ and
$\langle E_{ii} \otimes I, Y \rangle = 1$ have become redundant after
\FR in \cref{qap_sdp_facial}.

\index{$\GG$, gangster operator}
\index{gangster operator, $\GG$}

We now discuss the symmetry reduced program.
The symmetry reduced formulation of~\cref{qap_sdp} is studied in~\cite{de2012improved}.
We assume that the the automorphism group of the matrix $A$ is non-trivial.
To simplify the presentation, we assume
\[
A = \sum_{i=0}^{d}a_{i}A_{i},
\] where $\{A_{0},\ldots,A_{d}\}$ is the basis of the commutant of the automorphism group of $A$.
For instance the matrices $A_i$ ($i=0,1,\ldots,d$) may form a basis of
the Bose-Mesner algebra of the  Hamming scheme, see~\Cref{sec:Hamming}.
Further, we assume from now on that $A_0$ is a diagonal matrix, which is the case for  the Bose-Mesner algebra of the  Hamming scheme.
Here, we do not assume any structure in $B$. However the theory applies
also when $B$ has some symmetry structure
and/or $A_0$ is not diagonal; see our numerical tests for the minimum cut
problem in~\Cref{sec_mc}, below.

If the \SDP~\cref{qap_sdp} has an optimal solution $Y \in
\mathcal{S}^{n^2}_+$, then it has an optimal solution of the form $Y =\sum_{i=0}^{d} A_{i} \otimes Y_{i}$ for some matrix variables
$Y_{0},\ldots,Y_{d} \in \R^{n \times n}$, see~\eqref{AGforQAP} and \Cref{sec_group}.
We write these matrix variables in a more compact way as $y =
(\text{vec}(Y_{0}),\ldots,\text{vec}(Y_{d}))$, if necessary. Denote by
$\BBts_{k}(y) \in \mathcal{S}^{n_{k}}_+$ the $k$-th block of the block-diagonal matrix
\begin{equation}\label{qap_By}
\BBts(y) :=(Q \otimes I)^{T} Y (Q \otimes I) = \sum_{i=0}^{d}(Q^{T}A_{i}Q) \otimes Y_{i},
\end{equation}
where $Q$ is the orthogonal matrix block-diagonalizing $A_i$ ($i=0,\ldots, d$).

\begin{lemma}\label{lem_qap_sdp_s}
	The symmetry reduced program of the \DNN relaxation~\cref{qap_sdp} is given by
	\begin{equation}\label{qap_sdp_s}
	\begin{array}{cl}
	\min &  \sum_{i=0}^{d} a_{i}    \trace  (A_{i}A_{i})  \trace (BY_{i}) \\[1ex]
	\text{s.t.} & \sum_{i=0}^{d}  \trace (J A_{i})     \trace (J Y_{i} ) = n^2  \\[1ex]
            & \offDiag(Y_{0}) = 0 \\[1ex]
             & \diag(Y_{i}) = 0,\, i = 1,\ldots,d  \\[1ex]
		&   \diag (Y_{0}) =  \frac{1}{n}e_n \\[1ex]
	& Y_{j} \geq 0, j = 0,\ldots,d \\[1ex]
	& \BBts_{k}(y) \in \mathcal{S}^{n_{k}}_+, k = 1,\ldots,t,
	\end{array}
	\end{equation}
where $\BBts_{k}(y)$ is the $k$-th block from~\cref{qap_By}, and
\textdef{$\offDiag(Y_0)=0$}
is the linear constraints that the off-diagonal elements are zero.
\end{lemma}
\begin{proof}
	See e.g.,~\cite{truetsch2014semidefinite,MR2546331}.
\end{proof}

It remains to facially reduce the symmetry reduced program~\cref{qap_sdp_s}.
Note that  $W \in A_{\GG}$ can be written as $W = \sum_{i=0}^{d}
A_{i} \otimes W_{i}$, for some  matrices $W_{0},\ldots,W_{d} \in \R^{n \times n}$.
\Cref{thm:exposeG2} shows that the block-diagonal matrix
\begin{equation} \label{Wtilde}
\widetilde{W}:=(Q\otimes I)^{T}W(Q \otimes I) =  \sum_{i=0}^{d} (Q^{T}A_{i}Q) \otimes W_{i}
\end{equation}
is an exposing vector of the symmetry reduced program~\cref{qap_sdp_s}.
Further, we denote by $\widetilde{W}_k$ ($k=1,\ldots, t$) the $k$-th block of $\widetilde{W}$.
Let us illustrate this with~\Cref{qap_ham_ex}.
\begin{example}\label{qap_ham_ex}
Consider~\Cref{sec:Hamming}, where $A_i$ ($i=0,\ldots,d$) form a basis of the Bose-Mesner algebra of the Hamming scheme.
Then,  the exposing vector $W \in \mathcal{S}^{n^{2}}_+$ defined in
\Cref{qap_exp_lem} can be written as  $W = \sum_{i=0}^{d} A_{i} \otimes
W_{i}$,  where
\begin{equation}\label{qap_w_decomp}
W_{0} = (n-2)J + nI \text{ and } W_{i}  = nI_{n} - 2J \text{ for } i = 1,\ldots,d.
\end{equation}
Let $\widetilde{W}_{k} \in \mathcal{S}^{n}$ be the $k$-th block of $\widetilde{W}$, see~\cref{Wtilde}.
Then there  are $d+1$ distinct blocks given by $\widetilde{W}_{k} =
\sum_{i=0}^{d}p_{i,k}W_{i} \in \mathcal{S}^{n}$  for $k = 0,\ldots,d$,
where $p_{i,k}$ are elements in the character table $P$ of the Hamming scheme, see~\Cref{sec:Hamming}.
Using the fact that $Pe = (n,0,\ldots,0)^{T}$ and $p_{1,k} = 1$, for every $k = 0,\ldots,d$, we have
	\begin{equation}\label{qap_what_decomp}
	\widetilde{W}_{0} = n^{2}I - nJ  \text{ and } \widetilde{W}_{k} = nJ  \text{ for } k = 1,\ldots,d,
	\end{equation}
and the matrices $\tilde{V}_{k}$, whose columns form a basis of the null space of $\widetilde{W}_{k} \in \mathcal{S}^{n}$, are given by
\begin{equation}\label{qap_ham_Vhat}
\tilde{V}_{0} = e_n  \text{ and } \tilde{V}_{k} = \begin{bmatrix}
I_{n-1} \\
-e^T_{n-1}
\end{bmatrix} \in \R^{n \times (n-1)} \text{ for } k = 1,\ldots,d.
\end{equation}
\end{example}

Similar results can be derived when one uses different groups.
Now we are ready to present an \SDP relaxation for the \QAP that
is both facially and symmetry reduced.

\begin{prop}\label{lem_qap_fs}
	The facially reduced program of the symmetry reduced  \DNN  relaxation~\cref{qap_sdp_s} is given by
	\begin{equation}\label{qap_sdp_fs}
	\begin{array}{cl}
	\min &  \sum_{i=1}^{d} a_{i}    \trace  (A_{i}A_{i})  \trace (BY_{i}) \\[1ex]
	\text{s.t.} & \sum_{i=0}^{d}  \trace (J A_{i})     \trace (J Y_{i} ) = n^2 \\[1ex]
	& \offDiag(Y_{0}) = 0 \\ [1ex]
	& \diag(Y_{i}) = 0, i = 1,\ldots,d  \\	[1ex]
	& Y_{j} \geq 0, j = 0,\ldots,d \\[1ex]
	& \BBts_{k}(y) = \tilde{V}_{k}\tilde{R}_{k}\tilde{V}_{k}^{T},  k = 1,\ldots,t\\[1ex]
	& \tilde{R}_{k} \in \mathcal{S}^{n_{k}^\prime}_+, k = 1,\ldots,t.
	\end{array}
	\end{equation}
	Here, the columns of $\tilde{V}_{k} \in \R^{n_k \times
n_{k}^\prime}$ form a basis of the
 null space of the $\widetilde{W}_{k} \in \mathcal{S}^{n}$.
\end{prop}

\begin{proof}
	Applying~\Cref{thm:exposeG2} to the block-diagonal matrix
$(Q\otimes I)^{T}W(Q \otimes I) =  \sum_{i=0}^{d} (Q^{T}A_{i}Q) \otimes W_{i}$,
the matrices $\widetilde{W}_{k}$ are the exposing vectors of the symmetry reduced program~\cref{qap_sdp_s}, and thus
$\widetilde{W}_{k} \BBts_{k}(y)= 0$ for every $k = 1,\ldots,t$.
This means that there exists a full column rank matrix $\tilde{V}_{k}
\in \R^{n_k \times n_{k}^\prime} $ such that
$\BBts_{k}(y) = \tilde{V}_{k}\tilde{R}_{k}\tilde{V}_{k}^{T}$, where
$\tilde{R}_{k} \in \mathcal{S}^{n_{k}^\prime}_+$ for every $k = 1,\ldots,t$.
Finally, we apply~\Cref{redundant} to remove
redundant constraints, see also~\Cref{qap_facial}. This yields the
formulation~\cref{qap_sdp_fs}.
\end{proof}

Note that in the case that the basis elements $A_i$ ($i=0,\ldots, d$)
belong to the Hamming scheme, see~\Cref{qap_ham_ex},
it follows that $t=d+1$ in the above~\Cref{lem_qap_fs}.

\subsubsection{On solving \QAP with \ADMM}
\label{sect:solvqapadmm}
Now we discuss how to  use \ADMM to solve the \DNN relaxation~\cref{qap_sdp_fs},
for the particular case when $A_i$ ($i=0,1,\ldots,d$) form a basis of
the Bose-Mesner algebra of the  Hamming scheme.
We proceed as in~\Cref{sec_admm}, and exploit properties of the
known algebra, see~\Cref{sec:Hamming}.
Clearly, for any other algebra we can proceed in a similar way.
We assume without loss of generality that all the matrices
$\tilde{V}_{j}$ in this section have orthonormal columns.

First, we derive the  equivalent reformulation of the \DNN relaxation~\cref{qap_sdp_fs}, by exploiting the following.
\begin{enumerate}
	\item[(1)] Since we remove the repeating blocks of
positive semidefinite constraints,  to apply \ADMM we have to reformulate
the \DNN in such a way that~\Cref{psim_ass} is satisfied.
 Let us first  derive an expression for the objective function as follows.
	$$\begin{array}{rl}
	\trace ((A \otimes B) Y) & = \trace \big( (Q \otimes I)^{T}(\sum_{i=0}^{d}a_{i}A_{i} \otimes B) (Q \otimes I) (Q \otimes I)^{T}(\sum_{j=0}^{d} A_{j} \otimes Y_{j}) (Q \otimes I)\big) \\
	& =   \trace \bigg( \big( \sum_{i=0}^{d}(Q^{T}a_{i}A_{i}Q) \otimes B \big) \big( \sum_{j=0}^{d} (Q^{T}A_{j}Q) \otimes Y_{j} \big) \bigg) \\
	&=  \sum_{k=0}^{d} \mu_{k}   \trace  \big((\sum_{i=0}^{d} a_{i}  p_{i,k}B) ( \sum_{j=0}^{d} p_{j,k}Y_{j}) \big)  \\[1.5ex]
	& = \sum_{k=0}^{d} \langle \tilde{C}_{k}, \sqrt{\mu_{k}} \sum_{i=0}^{d} p_{i,k}Y_{i} \rangle,
	\end{array}$$
	where $\tilde{C}_{k} :=  \sqrt{\mu_{k}} (\sum_{i=0}^{d} a_{i}   p_{i,k}) B$.  Recall
	that $\mu = (\mu_{k}) \in \R^{d+1}$, with $\mu_k := {d \choose k}(q-1)^{k}$.
Then,  we multiply the coupling constraints $\BBts_{i}(y) = \tilde{V}_{i}\tilde{R}_{i}\tilde{V}_{i}^{T}$ by the square root of its multiplicities. Thus, for the Bose-Mesner algebra, we end up with $\sqrt{\mu_{j}}(\sum_{i=0}^{d}p_{i,j}Y_{i} -  \tilde{V}_{j}\tilde{R}_{j}\tilde{V}_{j}^{T}) = 0$.
	
\item[(2)] In many applications, it is not necessary to compute
high-precision solutions, and the \ADMM can be terminated at any
iteration. Then, one can use the dual variable
$\tilde{Z}_{j}$ from the current iteration to compute a valid lower
bound, see~\Cref{dual_bound}. By adding redundant constraints, this
lower bound is improved significantly when the \ADMM is terminated with low-precision.
Therefore we add the following  redundant constraints
\begin{equation}
Y_{0} = \frac{1}{n} I, \quad \trace(\tilde{R}_{j}) = \sqrt{\mu_{j}} p_{0,j} \text{ for } j = 0,\ldots,d.
\end{equation}
To see the redundancy of the last $d+1$ constraints above, we use the
fact that the columns of $\tilde{V}_{j}$ are orthonormal, and that
$\diag(Y_{i}) = 0, i=1,\ldots,d$, to derive
$$\trace(\tilde{R}_{j}) = \trace( \tilde{V}_{j}\tilde{R}_{j}\tilde{V}_{j}^{T}) = \trace\sqrt{\mu_{j}}(\sum_{i=0}^{d}p_{i,j}Y_{i} ) = \sqrt{\mu_{j}} p_{0,j}.$$
 This technique can also be found in~\cite{LiPongWolk:19,HaoWangPongWolk:14, hu2019solving,OliveiraWolkXu:15}. \label{ADMM_ingr}
\end{enumerate}
  We would like to emphasize that the techniques above are
not restricted to the Bose-Mesner algebra of the Hamming scheme.
Let us  present our reformulated \DNN relaxation for \ADMM. Define
\begin{equation}\label{admm_p}
\begin{array}{l}
\mathcal{P}:= \left\{ (Y_{0},\ldots,Y_{d}) \;|\;  \sum_{i=0}^{d}  {d
\choose i}(q-1)^{i}q^{d}  \trace (JY_{i}) = n^2,\;  \right.
\\\left. \qquad\qquad  \qquad \qquad Y_{0} = \frac{1}{n}
I,\; \diag(Y_{i}) = 0, Y_{j} \geq 0,\; i = 1,\ldots,d\right\},
\end{array}
\end{equation}
and
\begin{equation}\label{admm_R}
\mathcal{\tilde{R}}:= \{ (\tilde{R}_{0},\ldots,\tilde{R}_{d}) \;|\;  \trace(\tilde{R}_{j}) = \sqrt{\mu_{j}} p_{0,j},\; \tilde{R}_{i} \in \mathcal{S}^{n}_+, \; i = 0,\ldots,d\}.
\end{equation}
We obtain the following \DNN  relaxation for our \ADMM.
\begin{equation}\label{admm_qap}
\begin{array}{ccl}
p^{*}:=& \min &  \sum_{j=0}^{d} \langle \tilde{C}_{j},   \sqrt{\mu_{j}}\sum_{i=0}^{d} p_{i,j}Y_{i} \rangle  \\[1ex]
& \text{s.t.} & (Y_{0},\ldots,Y_{d}) \in \mathcal{P} \\[1ex]
& & (\tilde{R}_{0},\ldots,\tilde{R}_{d}) \in \mathcal{R}  \\[1ex]
& & \sqrt{\mu_{j}}(\sum_{i=0}^{d}p_{i,j}Y_{i} -  \tilde{V}_{j}\tilde{R}_{j}\tilde{V}_{j}^{T}) = 0,  j = 0,\ldots,d. \\[1ex]
\end{array}
\end{equation}

The augmented Lagrangian is
$$
\begin{array}{l}
{\LL}(\tilde{Y},\tilde{R},\tilde{Z}):=\sum_{j=0}^{d} \Big(  \langle \tilde{C}_{j},   \sqrt{\mu_{j}}\sum_{i=0}^{d} p_{i,j}Y_{i} \rangle  +   \langle \tilde{Z}_{j},  \sqrt{\mu_{j}}(\sum_{i=0}^{d}p_{i,j}Y_{i} -  \tilde{V}_{j}\tilde{R}_{j}\tilde{V}_{j}^{T}) \rangle
\\ \qquad \qquad \qquad  \qquad \qquad \qquad    \qquad +  \frac{\beta}{2} ||\sqrt{\mu_{j}}(\sum_{i=0}^{d}p_{i,j}Y_{i} - \tilde{V}_{j}\tilde{R}_{j}\tilde{V}_{j}^{T}) ||^{2} \Big).
\end{array}
$$
The ${Y}$-subproblem, the $\tilde{R}$-subproblem and the dual update are represented below.

\begin{enumerate}
	\item The ${Y}$-subproblem:
	\begin{equation}\label{qap_Ysub}
	\begin{array}{cl}
	\min & \sum_{j=0}^{d}||\sqrt{\mu_{j}} \sum_{i=0}^{d}p_{i,j}Y_{i} - \sqrt{\mu_{j}} \tilde{V}_{j}\tilde{R}_{j}\tilde{V}_{j}^{T} + \frac{\tilde{C}_{j}+ \tilde{Z}_{j}}{\beta}||^{2} \\[1.5ex]
	\text{s.t.} & Y_{0} = \frac{1}{n}I\\[1.5ex]
	& \diag(Y_{i}) = 0, ~i = 1,\ldots,d  \\[1.5ex]
	& \sum_{i=0}^{d}  {d \choose i}(q-1)^{i}q^{d}  \trace (JY_{i}) = n^2 \\[1.5ex]
	& Y_{i} \geq 0, ~i = 0,\ldots,d.
	\end{array}
	\end{equation}

	\item The $\tilde{R}$-subproblems, for $j = 0,\ldots,d$:
	\begin{equation}\label{qap_Rsub}
	\begin{array}{cl}
	\min & ||\tilde{R}_{j} - \tilde{V}_{j}^{T}(\sum_{i=0}^{d}p_{i,j}Y_{i} + \frac{\tilde{Z}_{j}}{\beta \sqrt{\mu_{j}}})\tilde{V}_{j}||^{2}  \\
	\text{s.t.} &  \tilde{R}_{j} \in \mathcal{S}^{n_{j}^\prime}_+.
	\end{array}
	\end{equation}
	
	\item Update the dual variable:
	\begin{equation}\label{qap_Zupdate}
	\tilde{Z}_{j} \leftarrow \tilde{Z}_{j} + \gamma \beta \sqrt{\mu_{j}} (\sum_{i=0}^{d}p_{i,j}Y_{i} - \tilde{V}_j\tilde{R}_j\tilde{V}_j^{T}), \quad j = 0,\ldots,d.
	\end{equation}
\end{enumerate}

Clearly, the $\tilde{R}$-subproblems can be solved in the same way as
\cref{admm_Rsub}.
To see that the ${Y}$-subproblem can also be solved efficiently, let us
show that it is a problem of the form~\cref{psim_def}, and thus
satisfies~\Cref{psim_ass}.

Let $\lambda_{j} = (p_{0,j},\ldots,p_{d,j})^{T} $,
$$y = \begin{bmatrix}
\text{vec}(Y_{0})\\
\vdots \\
\text{vec}(Y_{d})\\
\end{bmatrix} \text{ and } \hat{y} = \begin{bmatrix}
\text{vec}(\sqrt{\mu_{0}} \tilde V_{0} \tilde R_{0}\tilde V_{0}^{T} - \frac{\tilde{C}_{0} + \tilde Z_{0}}{\beta}) \\
\vdots \\
\text{vec}(\sqrt{\mu_{d}} \tilde V_{d} \tilde R_{d} \tilde V_{d}^{T} - \frac{\tilde{C}_{d} + \tilde Z_{d}}{\beta})
\end{bmatrix}.$$
Define the linear transformation ${\TT} ^*: \R^{(d+1)n^2} \rightarrow \R^{(d+1)n^2}$ by
$${\TT} ^*(y) = \begin{bmatrix}
\sqrt{\mu_{0}} \big( \lambda_{0}^{T} \otimes I_{n^{2}} \big) \\
\vdots \\
\sqrt{\mu_{d}} \big( \lambda_{d}^{T} \otimes I_{n^{2}} \big) \\
\end{bmatrix}y.$$

\begin{lemma}\label{qap_admm_ass}
	The ${Y}$-subproblem~\cref{qap_Ysub} is equivalent to the following projection to the weighted simplex problem
	\begin{equation}\label{qap_Ysub2}
	\begin{array}{cl}
	\min & ||{\TT} ^*(y) - \hat{y}||^{2} \\
	\text{s.t.} & y_{i} = 0, i \in \mathcal{I}\\
	& w^{T}y = n^2 \\
	& y \geq 0,
	\end{array}
	\end{equation}
	where $w := q^{d} (\mu \otimes e_{n^{2}}) \in \R^{(d+1)n^2}$, and $\mathcal{I}$ contains the indices of $y$ associated to
	the off-diagonal entries of $Y_{0}$. Furthermore, the problem
	\cref{qap_Ysub2} satisfies~\Cref{psim_ass}.
\end{lemma}
\begin{proof}
	One can verify that~\cref{qap_Ysub} and~\cref{qap_Ysub2} are equivalent. Furthermore, it holds that
	$$\begin{array}{cl}
	{\TT} ({\TT} ^*(y)) &= \begin{bmatrix}
	\sqrt{\mu_{0}} \big( \lambda_{0}^{T} \otimes I_{n^{2}} \big) \\
	\vdots \\
	\sqrt{\mu_{d}} \big( \lambda_{d}^{T} \otimes I_{n^{2}} \big) \\
	\end{bmatrix}^{T}\begin{bmatrix}
	\sqrt{\mu_{0}} \big( \lambda_{0}^{T} \otimes I_{n^{2}} \big) \\
	\vdots \\
	\sqrt{\mu_{d}} \big( \lambda_{d}^{T} \otimes I_{n^{2}} \big) \\
	\end{bmatrix}y \\
	& = \bigg(\sum_{j=0}^{d} \mu_{j} \big(  \lambda_{j}^{T} \otimes I_{n^{2}} \big)^{T} \big( \lambda_{j}^{T} \otimes I_{n^{2}} \big) \bigg)y \\
	& = \big((\sum_{j=0}^{d} \mu_{j} \lambda_{j}\lambda_{j}^{T})  \otimes I_{n^{2}} \big)y.
	\end{array}$$
	
	Applying the orthogonality relation of the Krawtchouk polynomial
	\cref{orth_kraw}, the $(r,s)$-th entry of $\sum_{j=0}^{d} \mu_{j}
	\lambda_{j}\lambda_{j}^{T}$ is $\sum_{j=0}^{d} \mu_{j}p_{r,j}p_{s,j} =
	q^{d}{d \choose s}(q-1)^{s} \delta_{r,s} = q^{d}\mu_{s} \delta_{r,s}$ for
	$r,s = 0,\ldots,d$. Thus ${\TT} ({\TT} ^*(y)) = \Diag(w)y$ and
	\Cref{psim_ass} is satisfied.
\end{proof}

To efficiently  solve the ${Y}$-subproblem for the \QAP, we
use~\Cref{prj_sim}. Finally we describe how to obtain a valid lower
bound when the \ADMM model is solved approximately.
The important problem of getting valid lower bounds from
inaccurate solvers is recently discussed in~\cite{Eckstein:2020}.

\begin{lemma}\label{dual_bound}
Let $\mathcal{P}$ be the feasible set defined in~\cref{admm_p}, and
consider the problem in~\cref{admm_qap}.
For any $\tilde{Z} = (\tilde{Z}_{0},\ldots,\tilde{Z}_{d})$, the
objective value
\begin{equation}\label{dual_obj}
\begin{array}{rcl}
g(\tilde{Z}) &:=& \min\limits_{(Y_{0},\ldots,Y_{d}) \in \mathcal{P}}\sum_{j=0}^{d} \langle \tilde{C}_{j} + \tilde{Z}_{j},   \sqrt{\mu_{j}}\sum_{i=0}^{d} p_{i,j}Y_{i} \rangle-\sum_{j=0}^{d}{\mu_{j}}p_{0,j} \lambda_{\max}(\tilde{V}_{j}^{T}\tilde{Z}_{j}\tilde{V}_{j})
\\&\leq & p^*,
\end{array}
\end{equation}
i.e.,~it provides a lower bound to the optimal value $p^{*}$ of~\cref{admm_qap}.
\end{lemma}
\begin{proof}
The dual of~\cref{admm_qap} with respect to the constraints $\sqrt{\mu_{j}}(\sum_{i=0}^{d}p_{i,j}Y_{i} -  \tilde{V}_{j}\tilde{R}_{j}\tilde{V}_{j}^{T}) = 0$ is
\begin{equation}\label{admm_d}
\begin{array}{ccl}
d^{*}:=\max\limits_{(\tilde{Z}_{0},\ldots,\tilde{Z}_{d})} &
\min\limits_{\stackrel{(Y_{0},\ldots,Y_{d}) \in \mathcal{P}}
{(\tilde{R}_{0},\ldots,\tilde{R}_{d}) \in \mathcal{R}}}
&  \sum_{j=0}^{d} \langle \tilde{C}_{j},   \sqrt{\mu_{j}}\sum_{i=0}^{d}
p_{i,j}Y_{i} \rangle + \langle \tilde{Z}_{j},
\sqrt{\mu_{j}}(\sum_{i=0}^{d}p_{i,j}Y_{i} -
\tilde{V}_{j}\tilde{R}_{j}\tilde{V}_{j}^{T}) \rangle.
\end{array}
\end{equation}
The inner minimization problem can be written as
\begin{equation}
\begin{array}{c}
\min\limits_{(Y_{0},\ldots,Y_{d}) \in \mathcal{P}}  \sum_{j=0}^{d}
\langle \tilde{C}_{j} + \tilde{Z}_{j},   \sqrt{\mu_{j}}\sum_{i=0}^{d}
p_{i,j}Y_{i} \rangle  +  \min\limits_{ (\tilde{R}_{0},\ldots,\tilde{R}_{d}) \in \mathcal{R}} \sum_{j=0}^{d}  \langle \tilde{Z}_{j},  \sqrt{\mu_{j}}( -  \tilde{V}_{j}\tilde{R}_{j}\tilde{V}_{j}^{T}) \rangle.
\end{array}
\end{equation}
It follows from the Rayleigh Principle, that the optimal value of the second minimization problem is $-\sum_{j=0}^{d}{\mu_{j}}p_{0,j} \lambda_{\max}(\tilde{V}_{j}^{T}\tilde{Z}_{j}\tilde{V}_{j})$. Using strong duality, we have $g(\tilde{Z}) \leq d^{*} = p^{*}$.
\end{proof}

\subsubsection{Numerical results for the \QAPp}
\label{sect:numericsQAP}

{
In this section we provide numerical results on solving the facially and
symmetry reduced  \DNN relaxation~\cref{qap_sdp_fs}.
We first present our general stopping conditions and tolerances
in~\Cref{def:tolerances}.

\begin{definition}[tolerances, stopping conditions]
\label{def:tolerances}
Given a \textdef{tolerance parameter, $\epsilon$}, we terminate the \ADMM when one of the following conditions is satisfied.
\index{$\epsilon$, tolerance parameter}
	\begin{itemize}
		\item The primal and dual residuals are smaller than $\epsilon$, i.e.,
		$$pres := \sum_{j=0}^{d}||\sum_{i=0}^{d}p_{i,j}Y_{i} -  \tilde{V}_{j}\tilde{R}_{j}\tilde{V}_{j}^{T}|| < \epsilon \text{ and } dres : = ||\tilde Z^{\text{old}} - \tilde Z^{\text{new}}||  \leq \epsilon.$$
		
\index{dres, dual residual}
\index{dual residual, dres}
\index{pres, primal residual}
\index{primal residual, pres}
		
		\item Let $p_{k}$ be the \ADMM objective value, and $d_{k} := g(\tilde{Z})$ the dual objective value
at some dual feasible point at the $k$-th iteration,
see~\cref{dual_obj}. If the duality gap is not improving significantly,
i.e.,
		$$
\textdef{gap $= \frac{p_{100k}-d_{100k}}{1+p_{100k}+d_{100k}}$} < 10^{-4},
$$
for $20$ consecutive integers $k$, then we conclude that there is
stagnation in the objective value. We measure the gap only every
$100$-th iteration {due to the expense of computing the dual objective value $d_{k}$.})
	\end{itemize}
\end{definition}

\index{\OBJp, objective \ADMM value}
\index{objective \ADMM value, \OBJp}
\index{\LBp, lower bound}
\index{lower bound, \LBp}

 In our \QAP experiments, we use $\epsilon = 10^{-12}$ if $n \leq 128$,
and $\epsilon = 10^{-5}$ when $n=256,512$. The objective value from
the \ADMM is denoted by \OBJ,
and the valid lower bound obtained from the dual feasible solution is
denoted by \LB, see~\Cref{dual_bound}.
The running times in all tables are reported in seconds. We also list
the maximum of the primal and dual residuals, i.e.,
\textdef{res~$:= \max\{pres,dres\}$}.
If a result is not available, we put \emph{-} in the corresponding entry.

\begin{enumerate}
\item
The first set of test instances are from  Mittelmann and Peng~\cite{mittelmann2010estimating},  where the authors compute \SDP bounds for the \QAP with $A$ being the Hamming distance matrix.
Choices of the matrix $B$\footnote{We thank Hans Mittelman for providing us generators for the mentioned instances.} differ for different types of instances.
In particular, in the Harper instance Harper\_$n$ where $n=2^d$ we have $B_{ij}=|i-j|$ for all $i,j=1,\ldots, 2^d$.
Further  eng1\_$n$  and end9\_$n$ with $n=2^d$, $d=4,\ldots, 9$ refer to the engineering problems, and  VQ\_$n$ instances have random matrices $B$.  For details see~\cite{mittelmann2010estimating}.
In rand\_256 and rand\_512 instances,  $A$ is the Hamming distance matrix of appropriate size and $B$ is a random matrix.

In the first column of \Cref{table:mittlem} we list the instance names where the sizes of the \QAP matrices  are indicated after the underscore.
Upper bounds are given in the column two.  For instances with up to 128 nodes we list the upper bounds computed in~\cite{mittelmann2010estimating},
and for the remaining instances we use our heuristics.
Since data matrices for the Harper instances are integer, we round up lower bounds to the closest integer.
In the column three (resp.~four) we list \SDP-based lower bounds (resp.~computation times in seconds) from~\cite{mittelmann2010estimating}.
The bounds from~\cite{mittelmann2010estimating} are obtained by solving an \SDP relaxation having several matrix variables on order $n$.
The bounds in~\cite{mittelmann2010estimating} were computed on a 2.67GHz Intel Core 2 computer with 4GB memory.
In the columns five to seven, we present the results obtained by using our \ADMM algorithm.

\Cref{table:mittlem}  shows that we significantly improve bounds for  all eng1\_$n$ and eng9\_$n$ instances.
 Moreover, we are able to compute bounds for huge \QAP instances with $n=256$ and $n=512$ in a reasonable amount of time.
Recall that for given $n$, the order of the matrix variable in the \DNN relaxation of the \QAP~\eqref{qap_sdp}  is $n^2$.
However,  for each instance $xx\_n$ of \Cref{table:mittlem}  we have that $n=2^d$, and that  the \DNN relaxation  \eqref{qap_sdp} boils down to $d+1$ positive semidefinite  blocks of order $n$.
In particular, we obtain the bound for each instance in \Cref{table:mittlem} by solving the facially and symmetry reduced  \DNN relaxation~\cref{qap_sdp_fs} where  $\BBts_{k}(y) \in \mathcal{S}^{n}_+$, $k = 1,\ldots,d+1$.

\begin{table}[H]
	\small
	\centering
	\begin{tabular}{|cc|cc|cccc|} \hline
		\multicolumn{2}{|c|}{} &
		\multicolumn{2}{c|}{MP~\cite{mittelmann2010estimating}}   &
		\multicolumn{4}{c|}{\ADMM}  \\  \hline
		problem  & UB & \LB & time & \OBJ & \LB & time  & res. \\  \hline
		Harper\_16 & 2752 & 2742 & 1 		   	& 2743 & 2742 & 1.92  & 4.50e-05   \\
		Harper\_32 & 27360 & 27328 & 3  		& 27331 & 27327 & 9.70  & 1.67e-04   \\
		Harper\_64 & 262260 & 262160 & 56		& 262196 & 261168 & 36.12  & 1.12e-05   \\
		Harper\_128 & 2479944 & 2446944 & 1491  & 2446800 & 2437880 & 186.12  & 3.86e-05   \\
		Harper\_256 & 22370940 & - & -    		& 22369996 & 22205236 & 432.10  & 9.58e-06   \\
		Harper\_512 & 201329908 & - & -    		& 201327683 & 200198783 & 1903.66  & 9.49e-06   \\ \hline
		eng1\_16 & 1.58049 & 1.5452 &  1 		& 1.5741& 1.5740 & 2.28  & 3.87e-05  \\
		eng1\_32 & 1.58528 & 1.24196 &  4		& 1.5669& 1.5637 & 14.63  & 5.32e-06  \\
		eng1\_64 & 1.58297 & 0.926658 & 56		& 1.5444& 1.5401 & 38.35  & 4.69e-06  \\
		eng1\_128 & 1.56962 & 0.881738 & 1688	& 1.4983& 1.4870 & 389.04  & 2.37e-06  \\
		eng1\_256 & 1.57995 & - & -    			& 1.4820& 1.3222 & 971.48  & 9.95e-06  \\
		eng1\_512 & 1.53431 & - & -    			& 1.4553& 1.3343 & 9220.13  & 9.66e-06  \\	\hline
		eng9\_16 & 1.02017 & 0.930857 &  1 	    & 1.0014& 1.0013 & 3.58  & 2.11e-06  \\
		eng9\_32 & 1.40941 & 1.03724 &  3 	    & 1.3507& 1.3490 & 12.67  & 3.80e-05  \\
		eng9\_64 & 1.43201 & 0.887776  & 68 	& 1.3534& 1.3489 & 74.89  & 6.60e-05  \\
		eng9\_128 & 1.43198 & 0.846574 & 2084 	& 1.3331& 1.3254 & 700.27  & 8.46e-06  \\
		eng9\_256 & 1.45132 & - & -    			& 1.3152& 1.2610 & 1752.72  & 9.74e-06  \\
		eng9\_512 & 1.45914 & - & -    			& 1.3074& 1.1168 & 23191.96  & 9.96e-06  \\	\hline
		VQ\_32 &  297.29 & 294.49   & 3			& 296.3241& 296.1351 & 11.82  & 1.27e-05  \\
		VQ\_64 & 353.5 & 352.4 &   45 			& 352.7621& 351.4358 & 43.17  & 4.22e-04  \\
		VQ\_128 & 399.09 & 393.29  &   2719 	& 398.4269& 396.2794 & 282.28  & 6.19e-04  \\	\hline
		rand\_256 & 126630.6273 & -  &	-		& 124589.4215 & 124469.2129 & 2054.61 & 3.78e-05 \\
		rand\_512 & 577604.8759 & - &	-		& 570935.1468 & 569915.3034 & 9694.71 & 1.32e-04 \\ \hline
	\end{tabular}
	\caption{Lower and upper bounds for different \QAP instances.}
	\label{table:mittlem}
\end{table}

\index{esc, Eschermann, Wunderlich}

\item
The second set of test instances are
\textdef{Eschermann, Wunderlich, esc}, instances from the QAPLIB library~\cite{MR1457185}.
In  esc\_$n$x instance, the  distance matrix $A$    is the Hamming distance matrix of order $n=2^d$,
whose automorphism group is the automorphism group of the Hamming graph $H(d,2)$.
In~\cite{MR2546331} the authors exploit symmetry in esc instances  to
solve the  \DNN relaxation~\cref{qap_sdp_s} by the  interior point method.
That was the \emph{first time} that \SDP bounds for large QAP instances
were computed by exploiting symmetry.
In particular, the authors from~\cite{MR2546331} needed  $13$ seconds to
compute the \SDP bound for esc64a,  and  $140$ seconds for computing the
esc128 \SDP bound,  see also~\Cref{table:Esc}.
The bounds in~\cite{MR2546331} are computed by the interior point solver SeDuMi~\cite{sturm2001using} using the Yalmip interface~\cite{lofberg2004yalmip} and Matlab 6.5,
implemented on a PC with Pentium IV 3.4	GHz dual-core processor and 3GB of memory.  Computational times in~\cite{MR2546331} include only solver time, not the time needed for  Yalmip to construct the problem.

In~\cite{OliveiraWolkXu:15} the authors approximately solve  the  \DNN
relaxation~\cref{qap_sdp_s} using the \ADMM  algorithm, but do note exploit symmetry.
Here, we compare computational results from~\cite{OliveiraWolkXu:15} with the approach we present in this paper.
All the instances from~\cite{OliveiraWolkXu:15} were tested on an Intel
Xeon Gold 6130 $2.10$ Ghz PC with $32$ cores and $64$ GB of
 memory and running on $64$-bit Ubuntu system.

An efficient solver, called {\bf SDPNAL$+$}, for solving large scale \SDPs is presented in~\cite{MR3384939,doi:10.1080/10556788.2019.1576176}.
 {\bf SDPNAL$+$}  implements an augmented Lagrangian based method.
 In particular, the implementation in~\cite{MR3384939} is  based on a majorized semismooth Newton-CG augmented Lagrangian method,
and  the implementation in~\cite{doi:10.1080/10556788.2019.1576176}  is based on an inexact symmetric Gauss-Seidel based semi-proximal \ADMM.
In~\cite{MR3384939,doi:10.1080/10556788.2019.1576176}, the authors
present extensive numerical results that also include solving~\cref{qap_sdp} on various instances from
the QAPLIB library~\cite{MR1457185}. However, they do not perform  \FR and \SRp.
In~\Cref{table:Esc} we include results from~\cite{doi:10.1080/10556788.2019.1576176} for solving  esc\_$n$x, with $n=16,32$.
There are no results for $n=64, 125$ presented in their paper. Moreover the authors emphasize that {\bf SDPNAL$+$}
 is for solving \SDPs where the maximum matrix dimension is
assumed to be less than $5000$. Due to the use of different computers,
the times in~\Cref{table:Esc} are not comparable. For example,
the authors from~\cite{doi:10.1080/10556788.2019.1576176}  use  an
Intel Xeon CPU E5-2680v3, $2.50$ GHz with $12$ cores and $128$ GB of memory.

In~\Cref{table:Esc} we present the numerical result for the esc instances.
In particular, we compare bounds and computational times  of the relaxation
\cref{qap_sdp} (no reductions, solved in~\cite{doi:10.1080/10556788.2019.1576176}),
the facially reduced relaxation~\cref{qap_sdp_facial} (solved in~\cite{OliveiraWolkXu:15}),
the symmetry reduced relaxation~\cref{qap_sdp_s} (solved in~\cite{MR2546331}),
and facially and symmetry reduced relaxation~\cref{qap_sdp_fs} (solved by our \ADMM).
\end{enumerate}

We conclude that:
\begin{enumerate}
	\item There are notably large differences in computation times between the \ADMM algorithm presented here and the one from~\cite{OliveiraWolkXu:15},
since the latter does not exploit symmetry.
\item \label{forRefPg2}
Even if the use of different computers is taken into account,
this would likely not be enough to account for the time differences observed
between our \ADMM  and {\bf SDPNAL$+$}~\cite{doi:10.1080/10556788.2019.1576176}.
Moreover,  {\bf SDPNAL$+$} was not able to solve several instances.

	\item In~\cite{MR2546331}, the authors use SeDuMi to solve a
relaxation equivalent to the symmetry reduced program~\cref{qap_sdp_s};
and they obtain a \LB of $53.0844$ for  esc128.
However, the bounds for this instance for the facially
and symmetry reduced program~\cref{qap_sdp_fs} computed by  the Mosek interior point method solver
 is $51.7516$; and our \ADMM algorithm reports $51.7518$.
 This illustrates our improved numerical accuracy using \FR and \SRp, and validates the statements about singularity degree,
see~\Cref{sect:singDeg}. We note in addition that we provide a theoretically guaranteed lower bound, as well as solve huge instances
that are intractable for the approach in~\cite{MR2546331}.

\end{enumerate}

\begin{table}[H]
\footnotesize
	\centering
	\begin{tabular}{|cc|cc|cc|cc|cccc|} \hline
		\multicolumn{2}{|c|}{} &  \multicolumn{2}{c|}{{\bf SDPNAL$+$}
		STYZ~\cite{doi:10.1080/10556788.2019.1576176}}  &
		\multicolumn{2}{c|}{\ADMM OWX~\cite{OliveiraWolkXu:15}}
		&  \multicolumn{2}{c|}{\SDP  KS~\cite{MR2546331}}  & \multicolumn{4}{c|}{\ADMM}  \\  \hline
		inst.  & opt & \LB & time  & \LB & time & \LB & time &
\OBJ & \LB & time & res  \\  \hline
		esc16a & 68 	& 63.2750 	& 16	& 64 	 & 20.14 	& 63.2756   & 0.75 		& 63.2856 	& 63.2856 	& 2.48 	& 1.17e-11  \\
		esc16b & 292 	& 289.9730 	& 24	& 290 	 & 3.10   	& 289.8817  & 1.04 		& 290.0000 	& 290.0000 	& 0.78 	& 9.95e-13  \\
		esc16c & 160 	& 153.9619 	& 65	& 154	 & 8.44   	& 153.8242  & 1.78 		& 154.0000 	& 153.9999 	& 2.11 	& 2.56e-09   \\
		esc16d & 16		& 13.0000  	& 2		& 13 	 & 17.39  	& 13.0000   & 0.89 		& 13.0000 	& 13.0000 	& 1.04 	& 9.94e-13  \\
		esc16e & 28		& 26.3367 	& 2		& 27	 & 24.04  	& 26.3368	& 0.51 		& 26.3368 	& 26.3368 	& 1.21 	& 9.89e-13  \\
		esc16f & 0 		& - 		& -  	& 0 	 & 3.22e+02 & 0			& 0.14 		& 0 			& 0 	& 0.01 	& 2.53e-14  \\
		esc16g & 26 	& 24.7388 	& 4 	& 25 	 & 33.54  	& 24.7403 	& 0.51 		& 24.7403 	& 24.7403 	& 1.40 	& 9.95e-13  \\
		esc16h & 996 	& 976.1857 	& 10 	& 977 	 & 4.01    	& 976.2244 	& 0.79 		& 976.2293 	& 976.2293 	& 2.51 	& 7.73e-13  \\
		esc16i & 14		& 11.3749 	& 6		& 12 	 & 100.79  	& 11.3749 	& 0.73 		& 11.3749 	& 11.3660 	& 6.15 	& 2.53e-06  \\
		esc16j & 8 		& 7.7938 	& 4 	& 8 	 & 56.90   	& 7.7942 	& 0.42 		& 7.7942 	& 7.7942 	& 0.21 	& 9.73e-13  \\
		esc32a & 130	& 103.3206 	& 333 	& 104    & 2.89e+03 & 103.3194 	& 114.88 	& 103.3211 	& 103.0465 	& 12.36 & 3.62e-06  \\
		esc32b & 168 	& 131.8532 	& 464 	& 132    & 2.52e+03 & 131.8718 	& 5.58 		& 131.8843 	& 131.8843 	& 4.64 	& 9.59e-13  \\
		esc32c & 642	& 615.1600 	& 331 	& 616    & 4.48e+02 & 615.1400 	& 3.70 		& 615.1813 	& 615.1813 	& 8.04 	& 2.05e-10  \\
		esc32d & 200	& 190.2273 	& 67 	& 191    & 8.68e+02 & 190.2266 	& 2.09 		& 190.2271 	& 190.2263 	& 5.86 	& 7.45e-08  \\
		esc32e & 2 		& 1.9001 	& 149 	& 2 	 & 1.81e+03 & - 		& -  		& 1.9000 	& 1.9000 	& 0.70 	& 4.49e-13  \\
		esc32f & 2		& - 		& - 	& 2 	 & 1.80e+03 & - 		& - 		& 1.9000 	& 1.9000 	& 0.76 	& 4.49e-13  \\
		esc32g & 6 		& 5.8336 	& 65 	& 6 	 & 6.04e+02 & 5.8330	& 1.80		& 5.8333 	& 5.8333 	& 3.50 	& 9.97e-13  \\
		esc32h & 438 	& 424.3256 	& 1076 	& 425 	 & 3.02e+03 & 424.3382 	& 7.16 		& 424.4027 	& 424.3184 	& 5.89 	& 1.03e-06  \\
		esc64a & 116 	& - 		& - 	& 98	 & 1.64e+04 & 97.7499 	& 12.99 	& 97.7500 	& 97.7500 	& 5.33 	& 8.95e-13  \\
		esc128 & 64  	& - 		& - 	& - 	 & -  		& 53.0844 	& 140.36 	& 51.7518 	& 51.7518 	& 137.71& 1.18e-12  \\ \hline
	\end{tabular}
	\caption{Esc instances (times with different computers).}
	\label{table:Esc}
\end{table}

\subsection{The graph partition problem (\GPp)}
\label{sec_mc}

The graph partition  problem  is the problem of partitioning the vertex set of a graph into a fixed number of sets of given sizes  such that the sum of edges joining different sets is  minimized.
The problem is known to be NP-hard.  The \GP has many applications such as VLSI design, parallel computing,
network partitioning, and floor planing.  Graph partitioning also plays
a role in machine learning (see e.g.,~\cite{li2015graph}) and data
analysis (see e.g.,~\cite{pirim2012clustering}). There exist several
\SDP relaxations for the \GP of different complexity and strength,  see
e.g.,
\cite{WoZh:96,sotirov2013efficient,MR3788893,van2015semidefinite,KaRe:95dec}.

\subsubsection{The general \GP}
\label{sect:generalGP}
Let $G=(V,E)$ be an undirected graph  with vertex set $V$, $|V|=n$ and edge set $E$, and $k\geq 2$ be a given integer.
We denote by $A$ the adjacency matrix of $G$.
The goal is to find a partition  of the vertex set into $k$ (disjoint) subsets $S_1,\ldots, S_k$ of
specified sizes $m_1\geq \ldots \geq m_k $, where $\sum_{j=1}^k m_j =n$, such that the sum of weights of edges joining different sets
$S_j$ is minimized.
Let
\begin{equation}\label{Pm}
P_m := \left \{ S=(S_1,\ldots, S_k)\,|\, S_i\subset V, |S_i|=m_i, \forall i, ~~S_i\cap S_j = \emptyset, i \neq j, ~\cup_{i=1}^k S_i =V \right \}
\end{equation}
denote the set of all partitions of $V$ for a given $m=(m_1,\ldots, m_k)$.
In order to model the \GP in binary variables we represent  the partition $S\in P_m$
by the partition matrix $X\in \R^{n\times k}$ where the column $j$ is the incidence vector for the set $S_j$.

The  \GP can be stated as follows
\[
 \min_{X\in {\mathcal M}_m}  \frac{1}{2} \trace (AX(J_k-I_k)X^T),
\]
where
\begin{equation} \label{Mm}
{\mathcal M}_m= \{ X\in \{0,1\}^{n\times k}\,|\, Xe_k= e_n, ~X^Te_n=m  \}
\end{equation}
is the set of partition matrices.

Here, we consider the following \DNN relaxation that is equivalent to the relaxation from~\cite{WoZh:96}:
\begin{equation}\label{mc_sdp}
\begin{array}{cl}
\min & \frac{1}{2}  \trace ( (A\otimes B) Y )\\
\text{s.t.} & \GG(Y) = 0 \\
&  \trace  (D_{1} Y) - 2(e_{n} \otimes e_{k})^{T}\diag(Y) + n = 0 \\
&  \trace  (D_{2} Y) - 2(e_{n} \otimes m)^{T}\diag(Y) + m^{T}m = 0 \\
& \mathcal{D}_{O}(Y) = \Diag(m) \\
& \mathcal{D}_{e}(Y) = e \\
& \langle J ,Y \rangle = n^2 \\
& Y\geq 0, Y \succeq 0,
\end{array}
\end{equation}
where $B=J_k-I_k$, and
$$Y = \begin{bmatrix}
{Y}^{(11)} &  \ldots & {Y}^{(1n)} \\
\vdots & \ddots & \vdots \\
{Y}^{(n1)} &  \ldots & {Y}^{(nn)} \\
\end{bmatrix}\in \mathcal{S}^{kn}$$ with each ${Y}^{(ij)}$ being a $k \times k$ matrix, and
$$\begin{array}{rl}
D_{1} &= I_{n} \otimes J_k \\
D_{2} &= J_n \otimes I_{k} \\
\mathcal{D}_{O}(Y) &= \sum_{i=1}^{n}Y^{ii} \in \mathcal{S}^{k} \\
\mathcal{D}_{e}(Y) & = (\trace Y^{ii}) \in  \R^{n} \\
\GG(Y) &= \langle I_{n} \otimes (J_k - I_{k}) , Y \rangle.
\end{array}$$
Here $\GG(\cdot)$ is the gangster operator for the \GP.
To compute \DNN bounds for the \GP, we apply \FR for symmetry reduced
relaxation~\cref{mc_sdp}.
 The details are similar  to the \QAP, and thus omitted.

\medskip
We present numerical results for different graphs  from the literature.
Matrix can161 is from the library Matrix Market~\cite{boisvert1997matrix}, matrix grid3dt5 is $3D$ cubical mesh, and gridt$xx$ matrices are $2D$ triangular meshes.
Myciel7 is a graph based on the Mycielski transformation and 1\_FullIns\_4 graph is a generalization of the Mycielski graph.
Both graphs are used in the COLOR02 symposium~\cite{color02}.

\begin{enumerate}
\item
In~\Cref{table:graphInfo}
\begin{table}[H]
	\small
	\centering
	\begin{tabular}{|c|c|c|c|c|} \hline                        
		instance & $|V|$ & \# orbits & blocks of $A$  &  $m$ \\  \hline
		1\_FullIns\_4 & 93 & 3629 & (53,27,9,3,1)& (30,31,32) \\
		can161 & 161 & 921 & (20,20,20,20,20,20,20,11,10)& (52,53,56) \\
		grid3dt5 & 125 & 4069 & (39,36,26,24)& (40,41,44) \\
		gridt15 & 120 & 2432 & (80,24,16)& (39,40,41) \\
		gridt17 & 153 & 3942 & (102,30,21)& (50,51,52) \\
		myciel7 & 191 & 6017 & (64,64,63)& (62,63,66) \\ \hline
	\end{tabular}
	\caption{Graphs and partitions.}
	\label{table:graphInfo}
\end{table}
we provide information on the graphs and the considered 3-partition problems. In particular,  the first column  specifies graphs,  the second column provides the number of vertices in a graph, the third column is the number of orbits after symmetrization,  the fourth column lists the number of blocks in $Q^T AQ$. Here, the orthogonal matrix $Q$ is computed by using the heuristic from~\cite{MR2546331}. The last column specifies sizes of partitions.

\item
In \Cref{table:GPPRezthree}
\begin{table}[H]
	\small
	\centering
	\begin{tabular}{|c|ccc|cccc|} \hline
		&  \multicolumn{3}{c|}{\begin{tabular}{@{}c@{}}\textbf{IPM} (Symmetry \\ \&Facially reduced)\end{tabular}} &  \multicolumn{4}{c|}{\ADMM $(\epsilon = 10^{-3})$}   \\  \hline
		instance &  \LB & time & iter. & \OBJ & \LB & time & res \\  \hline
		1\_FullIns\_4 	& 194.2825 & 311.95 & 26 & 194.2686 & 194.0523 & 141.29 & 1.50e-01 \\
		can161 	& 33.0151 & 124.32 & 19 & 33.0392 & 30.4470 & 19.74 & 2.58e-01 \\
		grid3dt5 	& 68.3175 & 245.65 & 17 & 68.3029 & 68.0436 & 200.35 & 2.02e-01 \\
		gridt15 	& 12.1153 & 1302.10 & 41 & 12.1116 & 11.3654 & 97.17 & 1.91e-01 \\
		gridt17* 	& 12.2482 & 1865.67 & 21 & 12.2532 & 11.1459 & 357.53 & 1.80e-01 \\
		myciel7* 	&  1126.0309 & 2579.65 & 17 & 1126.0385 & 1123.8526 & 553.67 & 9.50e-02 \\\hline
	\end{tabular}
	\caption{Numerical results for the graph $3$-partition.}
	\label{table:GPPRezthree}
\end{table}
we list lower bounds obtained by using Mosek and our \ADMM algorithm.
The table also presents computational times required to compute bounds
by both methods as well as the number of interior point method iterations.
The results show that the \ADMM with  precision $\epsilon = 10^{-3}$ provides competitive bounds in much shorter time than  the interior point method solver.
In~\Cref{table:GPPRezthree}, some instances are marked by $*$.
This means that our 64GB machine did not have enough memory to solve these instances by the interior point method solver,
and therefore they are solved on a machine with an Intel(R) Xeon(R) Gold 6126, 2.6 GHz quad-core processor and 192GB of memory.
However, the \ADMM algorithm has much lower memory requirements, and thus the \ADMM is able to solve all instances from~\Cref{table:GPPRezthree}
on the smaller machine.
\end{enumerate}

\subsubsection{The vertex separator problem ({\bf VSP}) and min-cut  ({\bf MC}) }
\label{sect:vertexsep}

The min-cut problem is the problem of partitioning the vertex set of a graph into $k$ subsets of given sizes
such that the number of edges joining the first $k-1$ partition sets is minimized.
The {\bf MC} problem is a special case of the general \GP but also arises as a subproblem of the vertex separator problem.
The vertex separator problem   is to find a subset of vertices (called vertex separator) whose removal disconnects the graph into $k-1$ components.
This problem is  NP-hard.

The vertex separator problem was studied by Helmberg, Mohar, Poljak and
Rendl~\cite{helmberg1995spectral}, Povh and Rendl~\cite{PoRe:05},
Rendl and Sotirov~\cite{rendl2018min}, Pong, Sun, Wang, Wolkowicz~\cite{HaoWangPongWolk:14}.
The {\bf VSP} appears in many different  fields such as   VLSI design~\cite{bhatt1984framework} and  bioinformatics~\cite{fu2005multi}.
Finding vertex separators of minimal size is an important problem in  communication networks~\cite{leighton1983complexity} and finite element methods~\cite{miller1998geometric}.
The VSP also appears in  divide-and-conquer algorithms   for minimizing the work involved in solving systems of equations, see e.g.,~\cite{lipton1980applications,lipton1979generalized}.

Let us formally relate the {\bf VSP} and the {\bf MC} problem.
Let $\delta(S_i,S_j)$ denote the set of edges between $S_i$ and $S_j$,
where $S_i$ and $S_j$ are defined as in~\cref{Pm}.
We denote the set of edges with endpoints in distinct partition sets $S_1$,\ldots,$S_{k-1}$ by
\[
\delta(S) = \cup_{i<j<k} \delta(S_i,S_j).
\]
The min-cut   problem is
\[
{\rm cut}(m) = \min \{ |\delta(S)| \,|\, S\in P_m \}.
\]
The graph has a vertex separator if there exists $S\in P_m$ such that after the removal of $S_k$ the induced
subgraph has no edges across $S_i$ and $S_j$ for $1\leq1<j<k$.
Thus, if ${\rm cut}(m)=0$ or equivalently $\delta(S)=\emptyset$, there exists a vertex separator.
On the other hand ${\rm cut}(m)> 0$  shows that no separator $S_k$ for the cardinalities specified in $m$ exists.

Clearly, $|\delta(S)|$ can be represented in terms of a quadratic function of the partition matrix $X$, i.e.,~as
$\frac{1}{2} \trace (AXBX^T)$ where
\begin{equation} \label{Bb}
B := \begin{bmatrix}
J_{k-1}-I_{k-1} & 0 \\
0 & 0
\end{bmatrix} \in \mathcal{S}^{k}.
\end{equation}
Therefore,
\[
{\rm cut}(m) = \min_{X\in {\mathcal M}_m}  \frac{1}{2} \trace (AXBX^T),
\]
where ${\mathcal M}_m$ is given in~\cref{Mm}.
To compute \DNN bounds for the {\bf MC} problem and provide bounds for the vertex separator problem,
we use the  \DNN relaxation~\cref{mc_sdp} with $B$ defined in~\cref{Bb}.

\begin{enumerate}
\item
We present numerical results for the Queen graphs, where the $n\times n$ Queen graph has the squares of an $n\times n$  chessboard for its vertices and two such vertices are adjacent if the corresponding squares are in the same row, column, or diagonal. The instances in this class come from the DIMACS challenge on graph coloring.  In~\Cref{table:QueenInfo} we provide information on the Queen graphs. The table is arranged in the same way as~\Cref{table:graphInfo}.

\begin{table}[H]
	\small
	\centering
	\begin{tabular}{|c|c|c|c|c|} \hline                        
	instance & $|V|$ &\# orbits &  blocks of $A$  & $m$   \\  \hline
	queen5\_5 & 25  & 91 & (12,6,3,3,1) & (4,5,16) \\
	queen6\_6 & 36  & 171 & (18,6,6,3,3) & (6,7,23) \\
	queen7\_7 & 49  & 325 & (24,10,6,6,3) & (9,9,31) \\
	queen8\_8 & 64  & 528 & (32,10,10,6,6) & (11,12,41) \\
	queen9\_9 & 81  & 861 & (40,15,10,10,6) & (14,15,52) \\
	queen10\_10 & 100  & 1275 & (50,15,15,10,10) & (18,18,64) \\
	queen11\_11 & 121  & 1891 & (60,21,15,15,10) & (21,22,78) \\
	queen12\_12 & 144  & 2628 & (72,21,21,15,15) & (25,26,93) \\
	queen13\_13 & 169  & 3655 & (84,28,21,21,15) & (30,30,109) \\\hline
	\end{tabular}
	\caption{The Queen graphs and partitions.}
	\label{table:QueenInfo}
\end{table}

\item
In~\Cref{table:VC} we provide the numerical results for the vertex separator problem.
 More specifically, we are computing the largest integer $m_3$ such that the solution value of the \DNN relaxation~\cref{mc_sdp}  is positive with partition
	\begin{equation}
	\label{mdef}
	m = (\lfloor \frac{n-m_3}{2} \rfloor, \lceil \frac{n-m_3}{2} \rceil, m_3).
	\end{equation}
Then $m_3+1$ is a lower bound for the vertex separator problem with respect to the choice of $m$.
One  may tend to solve~\cref{mc_sdp} for all possible $m_3$ between $0,1,\ldots,|V|-1$ to find the largest  $m_3$ for which the \DNN bound is positive.
However,  the optimal value of~\cref{mc_sdp} is monotone in $m_3$, and thus we find the appropriate $m_3$ using binary search starting with $m_3 = \lceil \frac{n}{2} \rceil$.
 We present the lower bound on the vertex separator, i.e.,
$m_3+1$ in the third column  of~\Cref{table:VC}. The total number of problems solved is listed in the fourth column of the same table.
The running time given in the last two columns is the total amount of
time used  to find a positive lower bound for~\cref{mc_sdp}  for some
$m_3$ by using Mosek  and  our \ADMM algorithm, respectively.
This task is particularly suitable for the \ADMM, as we can terminate
the \ADMM once the lower bound in an iterate is positive.
For example, it takes $786$ seconds to solve the min-cut relaxation on queen12\_12 by Mosek, see~\Cref{table:QueenRez}.
However, though not shown in the table, it takes \ADMM only $120$  seconds to conclude that the optimal value is positive.

\begin{table}[H]
	\small
	\centering
	\begin{tabular}{|c|c|c|c|c|c|c|} \hline                        
		instance & $|V|$ &  $m_3+1$ & $\#$problems &  \begin{tabular}{@{}c@{}}\textbf{IPM} (Symmetry\&Facially reduced) \\ time \end{tabular} & \begin{tabular}{@{}c@{}}\ADMM $(\epsilon = 10^{-12})$ \\ time \end{tabular}  \\  \hline
queen 5\_5 & 25 & 17 & 4 & 7.49 & 2.69 \\
queen 6\_6 & 36 & 24 & 5 & 9.62 & 2.91 \\
queen 7\_7 & 49 & 32 & 5 & 25.34 & 5.95 \\
queen 8\_8 & 64 & 42 & 6 & 85.72 & 34.35 \\
queen 9\_9 & 81 & 53 & 6 & 304.44 & 64.10 \\
queen 10\_10 & 100 & 65 & 7 & 1309.85 & 131.66 \\
queen 11\_11 & 121 & 79 & 7 & 3416.01 & 387.38 \\
queen 12\_12 & 144 & 94 & 7 & 6147.20 & 671.02 \\
queen 13\_13 & 169  & 110 & 8 & - & 1352.17 \\ \hline
	\end{tabular}
	\caption{The vertex separator problem on the Queen graphs.}
	\label{table:VC}
\end{table}

\item
In~\Cref{table:QueenRez} we compare bounds and computational times
required to solve, for fixed $m$, symmetry reduced \DNN relaxation~\cref{mc_sdp}
 by the interior point algorithm,
 as well as   symmetry and facially reduced relaxation~\cref{mc_sdp}  by using Mosek and our \ADMM algorithm.

\begin{table}[H]
	\small
	\centering
	\begin{tabular}{|c|ccc|ccc|cccc|} \hline
		&  \multicolumn{3}{c|}{  \begin{tabular}{@{}c@{}}\textbf{IPM} \\ (Symmetry reduced)\end{tabular}} & \multicolumn{3}{c|}{ \begin{tabular}{@{}c@{}}\textbf{IPM} (Symmetry \\ \&Facially reduced)\end{tabular} } & \multicolumn{4}{c|}{\ADMM $(\epsilon = 10^{-12})$}  \\  \hline
		instance & \LB & time & iter. & \LB & time & iter. &
\OBJ & \LB & time & res  \\  \hline
queen5\_5 & 0.0908 & 1.04 & 38 & 0.1658 & 0.27 & 10 & 0.1658 & 0.1658 & 6.88 & 7.36e-11    \\
queen6\_6 & 0.0962 & 3.43 & 31 & 0.1411 & 0.91 & 11 & 0.1411 & 0.1411 & 11.37 & 1.83e-10    \\
queen7\_7 & 0.5424 & 15.42 & 32 & 0.6196 & 1.92 & 10 & 0.6196 & 0.6196 & 17.97 & 5.53e-11    \\
queen8\_8 & 0.1967 & 127.60 & 39 & 0.3087 & 7.38 & 13 & 0.3087 & 0.3087 & 61.50 & 1.15e-10    \\
queen9\_9 & 0.0698 & 377.77 & 32 & 0.2175 & 19.98 & 12 & 0.2175 & 0.2175 & 204.39 & 1.16e-06    \\
queen10\_10 & 0.8159 & 1664.09 & 42 & 1.0211 & 85.42 & 14 & 1.0211 & 1.0211 & 239.75 & 1.09e-09    \\
queen11\_11 &   -   &    -    & -	 & 0.2131 & 275.20 & 16 & 0.2131 & 0.2131 & 877.85 & 1.82e-05    \\
queen12\_12 &   -   &    -    & -	 & 0.3248 & 786.12 & 25 & 0.3248 & 0.3248 & 1474.45 & 1.20e-06    \\
queen13\_13 &   -   &   -     & -	 &   -   &  -   & - & 0.9261 & 0.9261 & 1864.30 & 5.71e-09    \\ \hline
	\end{tabular}
	\caption{The min-cut problem on the Queen graphs.}
	\label{table:QueenRez}
\end{table}

\end{enumerate}

We conclude from~\Cref{table:VC,table:QueenRez} that
\begin{itemize}
	\item For small instances, the interior point algorithm  is faster than the \ADMM as shown in~\Cref{table:QueenRez}.
For larger instances, the interior point  algorithm has  memory issues.
 However, the \ADMM algorithm can still handle large instances due to its low memory demand.
	\item To obtain bounds on the vertex separator of a graph, one does not need to solve the \DNN relaxation to high-precision.
The \ADMM is able to exploit this fact, and find a lower bound on the size of the vertex separator in significantly less amount of time than the interior
point algorithm, see~\Cref{table:VC}.
	\item The symmetry reduced program without \FR
is heavily ill-conditioned, and the interior point method is not able to solve it correctly for any of the instances. The running time is also significantly longer than the symmetry and facially reduced program, see~\Cref{table:QueenRez}.

Note that we have solved the queen10\_10 problem with high accuracy
with \FR. The distance between the optimal solutions in
norm was very large with no decimals of accuracy.
This emphasizes the importance of \FR in obtaining
accuracy in solutions, see e.g.,~\cite{SWW:17}.
\end{itemize}

\section{Conclusion}
\label{sect:concl}

In this paper we propose a method to efficiently implement facial
reduction to the symmetry reduced \SDP relaxation,  and we
demonstrated the efficiency by solving large scale NP-hard problems.
More specifically, if an exposing vector of the
minimal face for the input \SDP is given, then we are able to construct
an exposing vector of the minimal face for the symmetry reduced \SDPp.
The resulting relaxation
is symmetry reduced, satisfies the Slater condition, and
thus can be solved with improved numerical stability.

We then extend our reduction technique to doubly nonnegative, \DNNp,
programs. In fact, our approach allows  the matrix variable of
the original \SDPp, to be passed to simple nonnegative vector for the
\DNNp. Again we exploit exposing vectors of \DNN as a decomposition into
a sum of a semidefinite and nonnegative exposing vectors.
Further, we discuss the importance of the order of the reductions in our theory.
We also show that the
singularity degree of the symmetry reduced program is equal to the  singularity degree of the original program.

We apply our technique to many combinatorial problems
and their \DNN relaxations, i.e.,~we facially and symmetry reduce
them.  The obtained relaxations can be solved extremely efficiently
using the alternating direction method of multipliers.
We also show that interior point methods  are more efficient on
a symmetry and facially reduced relaxation. As a result, we are able to compute improved lower bounds for some \QAP
instances in significantly less time.

	\printindex
	\label{ind:index}
	

\bibliography{../../../mytexmf/.master,../../../mytexmf/.edm,../../../mytexmf/.psd,../../../mytexmf/.bjorBOOK,../../../mytexmf/.qap,../../../mytexmf/.haoh}
	
\addcontentsline{toc}{section}{Bibliography}

\newpage

\end{document}